\documentclass[english,ruled]{article}
\usepackage[T1]{fontenc}
\usepackage[latin9]{inputenc}
\usepackage{verbatim}
\usepackage{subcaption}
\usepackage[ruled]{algorithm2e}
\usepackage{amsmath}
\usepackage{amsthm}
\usepackage{etoolbox}
\usepackage{amssymb}
\usepackage{wasysym}
\usepackage{graphicx}
\usepackage{xcolor}
\usepackage{mathtools}
\usepackage{multicol}
\usepackage{multirow}
\usepackage{geometry}
\usepackage{booktabs}
\usepackage{enumitem}
\usepackage{setspace}
\makeatletter
\usepackage[toc,page,header]{appendix}
\usepackage{minitoc}
\usepackage{ifthen}
\newboolean{doublecolumn}
\setboolean{doublecolumn}{false}
\newboolean{arxiv}
\setboolean{arxiv}{true}
\usepackage{pgfplots}
\usepackage{pdflscape}
\usepackage[pagebackref=true]{hyperref}
\usepackage{cleveref}
\usepackage{authblk}
\usepackage{todonotes}
\pgfplotsset{compat=1.15}

\newboolean{proofinpaper}
\setboolean{proofinpaper}{false}

\newcommand{\inputinpaper}[1]{%
  \ifbool{proofinpaper}{\begin{proof}
						\input{#1}{}%
						\end{proof}}}

\newcommand{\inputinappendix}[1]{%
\ifbool{proofinpaper}{}{\input{#1}}%
}

\makeatletter
\renewcommand\paragraph{%
  \@startsection{paragraph}{4}{\z@}%
    {-1.5ex\@plus -0.5ex \@minus -0.2ex}
    {-.5em \@plus -.1em}%
    {\normalfont\normalsize\bfseries}%
}
\makeatother
\newcommand{\assign}{\coloneqq}
\newcommand{\backassign}{=:}

\newcommand{\tmop}[1]{\ensuremath{\operatorname{#1}}}

\newcommand{\tmem}[1]{\textit{#1}}
\newcommand{\mathd}{\mathrm{d}}

\newcommand{\mathe}{\mathrm{e}}
\newcommand{\am}{\texttt{AM}}
\newcommand{\osms}{\texttt{OSMS}}
\newcommand{\pagd}{\texttt{PAGD}}

\definecolor{deepskyblue}{rgb}{0.0, 0.75, 1.0}

\global\long\def\vertiii#1{\left\vert \kern-0.25ex  \left\vert \kern-0.25ex  \left\vert #1\right\vert \kern-0.25ex  \right\vert \kern-0.25ex  \right\vert }%

\global\long\def\argmin{\operatornamewithlimits{arg\,min}}%

\global\long\def\diag{\mathrm{diag}}%

\global\long\def\and{\mathrm{and}}%

\global\long\def\Ebb{\mathbb{E}}%

\global\long\def\Rbb{\mathbb{R}}%

\global\long\def\Acal{\mathcal{A}}%

\global\long\def\Dcal{\mathcal{D}}%

\global\long\def\Ocal{\mathcal{O}}%

\renewcommand*\backref[1]{\ifx#1\relax \else (cited on #1) \fi}
\theoremstyle{plain}
\newtheorem{lem}{\protect\lemmaname}[section]
\theoremstyle{remark}
\newtheorem{rem}{\protect\remarkname}
\theoremstyle{plain}
\newtheorem{thm}{\protect\theoremname}[section]
\theoremstyle{plain}
\newtheorem{prop}{\protect\propositionname}[section]
\providecommand{\corollaryname}{Corollary}
\theoremstyle{plain}
\newtheorem{coro}{\protect\corollaryname}[section]
\theoremstyle{plain}

\theoremstyle{plain}

\providecommand{\lemmaname}{Lemma}
\providecommand{\remarkname}{Remark}
\providecommand{\theoremname}{Theorem}
\providecommand{\examplename}{Example}
\providecommand{\propositionname}{Proposition}
\providecommand{\definitionname}{Definition}

\newcommand{\1}{\mathbf{1}}
\newcommand{\us}{u^{\star}}
\newcommand{\vs}{v^{\star}}
\newcommand{\0}{\mathbf{0}}
\newcommand{\sk}{\texttt{SK}}

\crefdefaultlabelformat{#2\textbf{#1}#3} 
\crefname{section}{\textbf{section}}{\textbf{sections}}
\Crefname{section}{\textbf{Section}}{\textbf{Sections}}
\crefname{thm}{\textbf{Theorem}}{\textbf{theorems}}
\Crefname{thm}{\textbf{Theorem}}{\textbf{Theorems}}
\crefname{lem}{\textbf{Lemma}}{\textbf{lemmas}}
\Crefname{lem}{\textbf{Lemma}}{\textbf{Lemmas}}
\crefname{prop}{\textbf{proposition}}{\textbf{propositions}}
\Crefname{prop}{\textbf{Proposition}}{\textbf{Propositions}}
\crefname{algorithm}{\textbf{algorithm}}{\textbf{algorithms}}
\Crefname{algorithm}{\textbf{Algorithm}}{\textbf{Algorithms}}
\crefname{coro}{\textbf{Corollary}}{\textbf{corollaries}}
\Crefname{coro}{\textbf{Corollary}}{\textbf{corollaries}}
\crefname{definition}{\textbf{Definition}}{\textbf{definitions}}
\Crefname{definition}{\textbf{Definition}}{\textbf{definitions}}
\crefname{table}{\textbf{Table}}{\textbf{tables}}
\Crefname{table}{\textbf{Table}}{\textbf{tables}}
\crefname{figure}{\textbf{figure}}{\textbf{figures}}
\Crefname{figure}{\textbf{Figure}}{\textbf{Figures}}
\crefname{exple}{\textbf{example}}{\textbf{examples}}
\Crefname{exple}{\textbf{Example}}{\textbf{Examples}}

\crefname{rem}{\textbf{remark}}{\textbf{remark}}
\Crefname{rem}{\textbf{Remark}}{\textbf{Remarks}}

\begin{document}

\title{Tight Nonasymptotic Local Convergence of Sinkhorn-Knopp}

\author[1]{Wenzhi Gao\thanks{gwz@stanford.edu}}
\author[2]{Zhaonan Qu\thanks{zhaonan@google.com}}
\author[1,3]{Yinyu Ye\thanks{yyye@stanford.edu}}
\author[1,3]{Madeleine Udell\thanks{udell@stanford.edu}}
\affil[1]{ICME, Stanford University}
\affil[2]{Google}
\affil[3]{Department of Management Science and Engineering, Stanford University}

\maketitle

\begin{abstract}
We revisit the Sinkhorn-Knopp (\sk) algorithm for the matrix scaling problem. Despite extensive literature on the global convergence of {\sk} and its variants, its local linear convergence behavior remains less understood. We address this gap by providing the first nonasymptotic local analysis of {\sk} that matches the rate obtained from existing asymptotic Jacobian-based arguments. We show that under certain connectivity conditions, {\sk} is a polynomial-time algorithm for doubly stochastic matrix scaling. With the developed tools, we showcase the local suboptimality of {\sk} and provide accelerated variants. Finally, for dense matrices, we improve the complexity of existing first-order matrix scaling algorithms from $\Ocal(\tfrac{n^{7/3}}{\varepsilon^{2/3}})$ to $\Ocal(\tfrac{n^{9/4}}{\sqrt{\varepsilon}})$.
\end{abstract}

\section{Introduction} \label{sec:intro}

Given a nonnegative matrix $A \in \mathbb{R}^{m \times n}_+$ and two positive margin vectors $p \in \mathbb{R}^m_{+ +}, q \in \mathbb{R}^n_{+ +}$, the matrix scaling problem $(A, p, q)$ seeks two positive diagonal scaling matrices
$D_1^{\star}, D_2^{\star}$ such that the scaled matrix $A^{\star} \assign
D_1^{\star} A D_2^{\star}$ (approximately) satisfies the target margin conditions $A^{\star} \1_n = p$ and
$(A^{\star})^\top\1_m = q$. Matrix scaling arises in a wide range of applications, including computational optimal transport \cite{altschuler2017near,peyre2019computational}, numerical linear algebra \cite{knight2008sinkhorn}, choice modeling \cite{qu2025sinkhorn}, neural architecture design \cite{xie2025mhc}, and many others \cite{idel2016review}. \\

Among the algorithms for matrix scaling, Sinkhorn-Knopp ({\sk}) \cite{sinkhorn1967concerning} is arguably the simplest and most widely used. It is also known in the literature as the RAS algorithm \cite{bacharach1965estimating}. Algorithmically, {\sk} alternates between the two scaling matrices: it fixes one scaling matrix and updates the other so that the corresponding marginal condition is satisfied. Theoretically, the global worst-case iteration complexity of {\sk} is now well understood, following a sequence of works \cite{kalantari1996theorem,kalantari2008complexity,chakrabarty2021better,altschuler2017near,dvurechensky2018computational,he2026phase}. In terms of target accuracy $\varepsilon$, and ignoring the dimension dependence,
worst-case $\mathcal{O} ( \tfrac{1}{\varepsilon} )$ complexity
upper and lower bounds for {\sk} have been established in the
literature {\cite{dvurechensky2018computational,he2026phase}}.\\

Despite this conservative worst-case bound, {\sk} often reaches medium-to-high accuracy in only a few iterations in practice. In particular, it is frequently observed that {\sk} exhibits local linear convergence \cite{peyre2019computational}. On the one hand, a line of work has developed nonasymptotic global linear convergence guarantees for {\sk} \cite{franklin1989scaling,qu2025sinkhorn,he2026phase}. These rates are easily computable directly from the problem data $(A,p,q)$, but are generally not tight. On the other hand, by viewing {\sk} as a fixed-point iteration, the literature has obtained a much sharper local linear convergence rate via Jacobian linearization \cite{knight2008sinkhorn}. In particular, the asymptotic convergence rate is
\newpage
\begin{equation}
1 - \sigma_2
\assign
\lambda_{m-1}(
P^{-1/2} A^\star Q^{-1} (A^\star)^\top P^{-1/2}),
\label{eqn:tight-rate}
\end{equation}
where $P \assign \operatorname{diag}(p)$, $Q \assign \operatorname{diag}(q)$, and $\lambda_{m-1}$ denotes the second-largest eigenvalue. However, since this Jacobian-based argument applies in the limit as the iterates approach the optimum, it only yields an asymptotic guarantee. This gap motivates us to establish a nonasymptotic local linear convergence analysis whose rate matches the sharp asymptotic rate in \eqref{eqn:tight-rate}.

\paragraph{Contributions.}

\begin{itemize}[leftmargin=10pt]
\item We provide the first tight nonasymptotic local linear convergence
analysis of {\sk}, establishing an
$\mathcal{O}(c + \tfrac{1}{\sigma_2}\log(\tfrac{1}{\varepsilon}))$
iteration complexity. In particular, by explicitly computing the
problem-dependent constant $c$, we show that {\sk} is a polynomial-time algorithm for the doubly-stochastic matrix scaling problem when
$\sigma_2$ is treated as a fixed constant. Our analysis also provides a
general template for proving nonasymptotic local convergence rates of
alternating minimization algorithms.

\item Building on these tools, we derive accelerated variants of {\sk}
that achieve improved nonasymptotic global and local complexity bounds for
the matrix scaling problem.
\end{itemize}

\subsection{Related literature}

There is a vast literature on the matrix scaling problem. We focus on recent
complexity-theoretic advances and refer interested readers to
\cite{idel2016review} for a comprehensive review.

\paragraph{Matrix scaling and balancing.}Matrix scaling is a widely studied
problem in numerical linear algebra {\cite{knight2008sinkhorn}} and
optimization {\cite{diamond2017stochastic}}. Early works on the complexity of
the matrix scaling problem treat the problem as a structured convex
problem and obtain polynomial-time algorithms that depend on $\log \| A
\|_{\infty}, \log \| A \|_{- \infty}, \log \| p \|_1$ and $\log (
\tfrac{1}{\varepsilon} )$
{\cite{nemirovski1999complexity,kalantari1996complexity}}.
{\cite{linial1998deterministic}} further removes the dependence on $\log \| A
\|_{- \infty}$ and obtains a strongly polynomial-time algorithm. Recent
advances on the matrix scaling problem, including
{\cite{allen2017much,cohen2017matrix}}, develop first- and second-order
methods for the matrix scaling problem. The matrix scaling problem is also
closely related to matrix balancing and equilibration. See
{\cite{knight2008sinkhorn,diamond2017stochastic}} for a more detailed
discussion of the relationships between these problems.

\paragraph{Sublinear convergence of Sinkhorn-Knopp.} The global sublinear
 convergence of {\sk} was first established in \cite{kalantari1993rate,kalantari2008complexity} and later improved in a sequence of works
\cite{altschuler2017near,dvurechensky2018computational,chakrabarty2021better}. The current state-of-the-art complexity of {\sk} is
$\mathcal{O} ( \tfrac{D}{\varepsilon} )$, where $D$ is an upper
bound on the minimum-norm solution pair in $\log$ space. Recently, 
\cite{he2026phase} establishes an $\Omega ( \tfrac{\sqrt{n}}{\varepsilon} )$
iteration complexity lower bound. Hence, the two bounds match in
terms of the order of $\varepsilon$.

\paragraph{Linear convergence of Sinkhorn-Knopp.} This line of work aims to
establish linear convergence of {\sk}. Some of the results provide
global linear convergence rates: for example, \cite{franklin1989scaling}
establishes a linear convergence rate $\tfrac{\sqrt{\kappa} - 1}{\sqrt{\kappa}
+ 1}$ based on Hilbert's projective metric, where $\kappa = \max_{i, j, k, l}
\frac{a_{i k} a_{j l}}{a_{j k} a_{i l}}$. Recently, \cite{qu2025sinkhorn}
establishes a linear convergence rate based on the second smallest eigenvalue
of $\big(\begin{smallmatrix}
  P & A\\
  A^{\top} & Q
\end{smallmatrix}\big)$. Another notable recent result is \cite{he2026phase},
where the authors define a density parameter based on the normalized version
of the matrix and show linear convergence when this parameter is fixed.
However, the above results are often conservative upper bounds on the behavior
of {\sk}, and some only apply to strictly positive matrices.
In contrast, another line of research obtains tight contraction factors based
on local arguments \cite{knight2008sinkhorn,peyre2019computational}. These
arguments treat {\sk} as a fixed-point iteration and obtain the tight
local linear convergence rate in \eqref{eqn:tight-rate} by analyzing the spectrum of the Jacobian. The local rate matches the true performance of {\sk}, but such arguments
are only asymptotic. This paper closes this gap by providing a nonasymptotic
local analysis for {\sk} for nonnegative matrices.
\section{Sinkhorn-Knopp algorithm for matrix scaling} \label{sec:preliminary}

\paragraph{Notation.}Throughout the paper, we use $\langle \cdot, \cdot
\rangle$ to denote the Euclidean inner product and $\| \cdot \|$ to denote the
Euclidean norm. We denote $\|A\|_1 \assign \sum_{i ,j} |a_{ij}|$ and $\|A\|_{-\infty} \assign \min_{|a_{i j}| > 0} |a_{i j}|$. Given a vector $d \in \mathbb{R}^n$, we use $\mathcal{D} (d) = \diag(d)$
to denote a corresponding diagonal matrix with $d$ on its diagonal. \ The notation
$\1_n$ denotes the all-ones vector of dimension $n$. We will frequently use the
notation $U =\mathcal{D} (u)$ and $V =\mathcal{D} (v)$. The notation $\exp
(\cdot) = \mathe^{(\cdot)}$ and $\log(\cdot)$ will be applied element-wise to a vector or to
the nonzeros of a matrix, i.e., $\exp (U) =\mathcal{D} (\exp (u))$. Given two vectors $a, b$ of the same dimension, we
use $a / b =\mathcal{D} (b)^{- 1} a$ to denote the element-wise division. 
Given a symmetric matrix $A$, $\lambda_k(A)$ denotes its $k$-th smallest eigenvalue.
We often use $(a, b) \assign (\begin{smallmatrix}a \\b\end{smallmatrix})$ to denote the concatenation of column vectors when the context is clear. We denote $P =\mathcal{D} (p),
Q =\mathcal{D} (q)$ and $S = \left(\begin{smallmatrix}
  P & \\
  & Q
\end{smallmatrix}\right)$.

\subsection{Matrix scaling and Sinkhorn-Knopp iteration}
We formally define the matrix scaling problem: given a nonnegative matrix $A \in \mathbb{R}^{m \times n}_+$ and two target
margins $p \in \mathbb{R}^m_{+ +}, q \in \mathbb{R}^n_{+ +}$ such that $\langle \1_m, p\rangle= \langle \1_n, q \rangle$, the matrix
scaling problem looks for two diagonal scaling matrices $D_1, D_2$ such that two margin conditions $D_1 A D_2 \1_n = p,  D_2 A^{\top} D_1 \1_m  = q$
are approximately satisfied:
\begin{equation}\label{eqn:approx-margin}
	\max \{ \| D_1 A D_2 \1_n - p \|, \| D_2 A^\top D_1 \1_m- q \|
   \} \leq \varepsilon .
\end{equation}
\begin{rem}
In the recent literature on matrix scaling, the residual of the margin condition is often measured in the $\ell_1$-norm when $p, q$ are probability margins. Since our focus is linear convergence, we will adopt $\ell_2$-error, and converting linear convergence to $\ell_1$ error loses a $\log n$ term.
\end{rem}
A pair $(D_1, D_2)$ satisfying \eqref{eqn:approx-margin}  is called an $\varepsilon$-approximate scaling.
Without loss of generality, we reparametrize $D_1, D_2$ by $D_1 = \mathe^U$ and
$D_2 = \mathe^{- V}$ for real diagonal matrices $U, V$. {\sk} alternates between the two scaling matrices by fixing one and forcing the other to satisfy the margin condition:
\[D_1^{k+1} = \mathcal{D}(p / (A D_2^k \1_n)) \quad \text{and} \quad  D_2^{k+1} = \mathcal{D}(q / (A^\top D_1^{k+1} \1_m)).\]
%
Under mild conditions \cite{idel2016review}, this iteration will converge to an optimal
scaling pair $(D_1^{\star}, D_2^{\star})$. Without loss of generality, we assume that such a pair exists and is finite, i.e., the matrix $A$ is scalable with respect to $p, q$.

\begin{enumerate}[leftmargin=25pt,itemsep=2pt,label=\textbf{A\arabic*:},ref=\rm{\textbf{A\arabic*}},start=1]
\item The nonnegative matrix $A$ is scalable with respect to margins $p, q$, and $\sigma_2 > 0$.
\label{A1}
\end{enumerate}

\paragraph{Dual potential.}A useful technique for analyzing the Sinkhorn iteration is through the potential function \cite{idel2016review}
\begin{equation} \label{eqn:sk-potential}
	\varphi (u, v) \assign \langle \mathe^u, A \mathe^{- v} \rangle - \langle
   p, u \rangle + \langle q, v \rangle,
\end{equation}
and {\sk} can be written as performing alternating minimization (\am) on it:
\[ \quad u^{k + 1} = \argmin_u
   ~\varphi (u, v^{k}),\quad v^{k + 1} = \argmin_v ~\varphi (u^{k+1}, v). \]
   
This paper adopts the {\am} perspective and will stick to \Cref{alg:sinkhorn-altmin}.
\begin{algorithm}[h]
{\textbf{input} $(A, p, q)$, initial $v^1$ (or $u^1$)}

\For	{$k = 1, 2, \dots$}{
$u^{k+1} = \argmin_u
   ~\varphi (u, v^{k})$\\
$v^{k+1} = \argmin_v ~\varphi (u^{k+1}, v)$
}
\caption{Sinkhorn-Knopp ({\am} on the potential function\label{alg:sinkhorn-altmin} \eqref{eqn:sk-potential})}
\end{algorithm}

\section{Global and local convergence of Sinkhorn-Knopp} \label{sec:global}

This section presents our main result on the convergence of the Sinkhorn algorithm.

\subsection{Algorithm analysis}

\paragraph{Global convergence.} 
The local behavior of an algorithm typically follows a phase of global convergence. Define the solution norm diameter constant
\begin{equation} \label{eqn:diam}
	D \assign \min_{( \us, \vs ) \in \argmin \varphi}  \|
   ( \us, \vs ) \|_{\infty} .
\end{equation}
Under \ref{A1}, there exists a finite scaling pair $(\us, \vs)$, and the constant $D < \infty$ can be explicitly bounded under additional assumptions:
\begin{lem}[\cite{lin2019efficient,kalantari2008complexity,allen2017much}]
  \label{lem:bounded-norm}Suppose \ref{A1} holds. Then
  \begin{itemize}[leftmargin=10pt]
    \item Square doubly-stochastic. If $p = q = \tfrac{1}{n} \1_n$, then $D \leq \tfrac{1}{2} |
    \log ( \tfrac{1}{n \| A \|_{- \infty}} ) | + n \log
    ( \tfrac{\| A \|_1}{n \| A \|_{- \infty}} )$,
    
    \item Doubly-stochastic. If $p = \tfrac{1}{m} \1_m$, $q = \tfrac{1}{n} \1_n$, and $m
    \leq n$, then $D \leq \tfrac{1}{2} | \log ( \tfrac{1}{n \| A
    \|_{- \infty}} ) | + n^2 \log ( \tfrac{\| A \|_1}{n \| A
    \|_{- \infty}} )$,
    
    \item Positive matrix. If $A > 0$ and $\langle \1_m, p \rangle = \langle
    \1_n, q \rangle = 1$, then $D \leq \log ( \tfrac{n}{\| A
    \|_{- \infty}  \| (p, q) \|_{- \infty}^2} )$.
  \end{itemize}
\end{lem}

Given \Cref{lem:bounded-norm}, an explicit convergence rate for the suboptimality of $\varphi $ is immediate.

\begin{thm}
  \label{thm:global} Suppose \ref{A1} holds and let $(u^k, v^k)$ be
  generated by \Cref{alg:sinkhorn-altmin}. Then
  \[ \varphi (u^{K + 1}, v^{K + 1}) - \varphi ( \us, \vs ) \leq
     \tfrac{2 \|p\|_1 D^2}{K}, \]
  where the diameter $D$ is defined in \eqref{eqn:diam}.
\end{thm}

\paragraph{Local convergence.}
Given \Cref{thm:global}, the function value gap will vanish and finally enter the local convergence regime. Our local analysis hinges on a simple
relation, given by \Cref{prop-equivalence} below.

\begin{prop}
  \label{prop-equivalence}Suppose \ref{A1} holds. Then we have the
  following identity:
  \[ \textstyle \varphi (u, v) = \varphi (u^{\star}, v^{\star}) + \sum_{i = 1}^m \sum_{j
     = 1}^n a_{i j}^{\star} \phi (\Delta_{i j}), \]
  where $\phi (\delta) \assign \mathe^{\delta} - \delta - 1$, $A^{\star}$ is
  the optimally scaled matrix, and $\Delta_{i j} \assign (u_i - u_i^{\star}) -
  (v_j - v_j^{\star})$.
\end{prop}

\Cref{prop-equivalence} suggests we could analyze the potential function
$\varphi$ through the scalar function
\[ \phi (\delta) = \mathe^{\delta} - \delta - 1 = \frac{\delta^2}{2} + \delta^3
   \int_0^1 \frac{(1 - t)^2}{2} \mathe^{\delta t} ~\mathd t, \]
whose local behavior is dictated by $\frac{\delta^2}{2}$ and higher-order remainder terms. In particular, it is convenient to bound the first- and second-order information with zeroth-order information:

\begin{lem}
  \label{lem:phi}Let $\phi (\delta) = \mathe^{\delta} - \delta - 1$. Then
  $\phi' (\delta) - \delta = \phi (\delta)$ and $| \phi' (\delta) | = | \phi''
  (\delta) - 1 | \leq \sqrt{2 \phi (\delta)} + \phi (\delta)$.
\end{lem}

As a standard component of local analysis, we introduce the second-order expansion that governs the local behavior. Define $\lambda (u, v) \assign
\varphi (u^{\star}, v^{\star}) + \tfrac{1}{2} \| (\Delta u, \Delta v)
\|_{\nabla^2 \varphi (u^{\star}, v^{\star})}^2$. By definition, we have 
\[
\textstyle  \lambda(u, v) = \varphi (u^{\star}, v^{\star}) + \sum_{i, j} \tfrac{a_{i
  j}^{\star}}{2} \Delta_{i j}^2,\quad \nabla \lambda (u, v)= \nabla^2 \varphi (u^{\star}, v^{\star}) (\Delta u, \Delta v),\quad \text{and} \quad\nabla^2 \lambda (u, v) = \nabla^2 \varphi (u^{\star}, v^{\star}).
\]
Since we will mostly focus on single-step progress, from now on, we will resort to the notation 
\[ (u, v) ~\rightarrow~ (u^+, v)~\rightarrow~(u^+, v^+)\]
to denote one iteration of {\sk} (or {\am}).
A useful fact is that {\am} on $\lambda (u, v)$
coincides with preconditioned gradient descent with preconditioners
$P^{- 1}$ and $Q^{- 1}$, whose contraction can be explicitly computed:
\begin{lem}
  \label{lem:contraction}Let $(u, v), (u^+, v)$ and $(u^+, v^+)$ be consecutive iterations obtained by running {\am} on $\lambda$. Then
  \[ \lambda (u^+, v^+) - \varphi (u^{\star}, v^{\star}) \leq (1 - \sigma_2)
     [\lambda (u^+, v) - \varphi (u^{\star}, v^{\star})] \leq (1 - \sigma_2)^2
     [\lambda (u, v) - \varphi (u^{\star}, v^{\star})],\]
 where $\sigma_2 = \lambda_2 ( I - P^{- 1 / 2} A^{\star}
  Q^{- 1} ({A^{\star}})^{\top} P^{- 1 / 2} ) \in (0, 1]$ is the
  connectivity of the normalized Laplacian.
\end{lem}

The quantity $\sigma_2$ already appeared in the {\sk}
literature {\cite{peyre2019computational,knight2008sinkhorn,qu2025sinkhorn}},
and it is obtained by treating {\sk} as a fixed-point iteration and
linearizing its Jacobian. This quantity tightly characterizes the asymptotic behavior of the algorithm and is therefore a desirable measure of local convergence. Given that $\varphi$ behaves like $\lambda$ in the local regime, and that {\am} applied to the quadratic function $\lambda$
produces a contraction of $1 - \sigma_2$, it remains to connect $\varphi$ to
$\lambda$. With the help of \Cref{lem:phi},  \Cref{lem:grad-hess-new} below makes this connection explicit.

\begin{lem}
  \label{lem:grad-hess-new}Suppose \ref{A1} holds and let $\varepsilon =\varphi (u, v) - \varphi
  (u^{\star}, v^{\star})> 0$. Then we have
  \begin{enumerate}[leftmargin=15pt]
    \item $\| \nabla \varphi (u, v) \|_{S^{- 1}} \leq 2 \sqrt{\varepsilon} +
    \tfrac{2}{\sqrt{s}} \varepsilon$ and $\| \nabla \varphi (u, v) -
    \nabla \lambda (u, v) \|_{S^{- 1}} \leq \tfrac{2}{\sqrt{s}} \varepsilon$, where $s \assign \| (p, q) \|_{- \infty}$;
    
    \item $\| S^{- 1 / 2} [\nabla^2 \varphi (u, v) - \nabla^2 \lambda (u, v)]
    S^{- 1 / 2} \| \leq 4 \sqrt{\frac{\varepsilon}{s}} + \tfrac{2\varepsilon}{s}
    $, and recall that $S = \left(\begin{smallmatrix}
  P & \\
  & Q
\end{smallmatrix}\right)$;
    
    \item if  $\varepsilon \leq 
    \tfrac{s \sigma_2^2}{1296}$, then $| \varphi (u, v) - \lambda (u,
    v) | \leq  \tfrac{168}{\sigma_2} (\tfrac{1}{\sqrt{s}} \varepsilon^{3 / 2} +
   \tfrac{1}{s} \varepsilon^2 )$.
  \end{enumerate}
\end{lem}
Using \Cref{lem:grad-hess-new}, we obtain a local recursion on the potential function gap.
\begin{lem}
  \label{lem:global-conv}
Suppose \ref{A1} holds and $\nabla_v \varphi (u, v) = 0$. If $\varphi (u, v) - \varphi (u^{\star},
  v^{\star}) \leq 
    \tfrac{s \sigma_2^2}{1296}$ with $s \assign \| (p, q) \|_{- \infty}$, then
  \begin{align}
     \varphi (u^+, v) - \varphi (u^{\star}, v^{\star}) 
    \leq{} & (1 - \sigma_2) [\varphi (u, v) - \varphi (u^{\star}, v^{\star})] +
    \tfrac{336}{\sqrt{s}\sigma_2}  [\varphi (u, v) - \varphi (u^{\star},
  v^{\star})]^{3 / 2} + \tfrac{338}{s \sigma_2} [\varphi (u, v) - \varphi (u^{\star},
  v^{\star})]^2,\nonumber
  \end{align}
and the same result holds for $\varphi (u^+, v^+) - \varphi (u^{\star},
  v^{\star})$ with respect to $\varphi (u^+, v) - \varphi (u^{\star},
  v^{\star})$.
\end{lem}

We are ready to present the nonasymptotic complexity of Sinkhorn that only depends on $(A, p, q)$ and $\sigma_2$. 

\begin{thm}
  \label{thm:global-conv}Suppose \ref{A1} holds and  $\varepsilon \in (0, 16 \| p \|_1]$. Then {\sk} finds an $\varepsilon$-approximate scaling in
  \[ K_{\varepsilon} = \Big\lceil \tfrac{2 \cdot 1344^2 \|p\|_1 D^2}{ s \sigma_2^4} + \tfrac{1}{\sigma_2} \log (128 [\| A \|_1 \| p \|_1 + \|
     p \|_1^2] D) + \tfrac{2}{\sigma_2} \log ( \tfrac{1}{\varepsilon}
     ) \Big\rceil = \Ocal\big(\tfrac{\|p\|_1 D^2} {s \sigma^4_2} + \tfrac{1}{\sigma_2} \log(\tfrac{1}{\varepsilon})\big) \]
  iterations, where the solution norm bound $D$ is defined in \eqref{eqn:diam}, $\|A\|_1 = \sum_{i, j} a_{i j}$, and  $s = {\|(p, q)\|_{-\infty}}$.
\end{thm}

\subsection{Implications and extensions}

\paragraph{Tightness.} Since $\sigma_2$ has been shown to match the rate obtained by linearizing the Jacobian, it is easy to find instances where our analysis matches the asymptotic behavior of the algorithm. 

\begin{prop}[Informal] \label{thm: tight}There exists (a family of) matrix scaling instances $\{(A, p, q)\}$ such that $\varphi(u^k, v^k) - \varphi(\us, \vs) \geq \varepsilon$ for $k = \Ocal(c+ \frac{1}{\sigma_2} \log(\frac{1}{\varepsilon}))$, where $c\geq 0$ does not depend on $\varepsilon$.	
\end{prop}

\paragraph{Polynomiality of {\sk}.} For doubly stochastic matrix scaling, our analysis shows that {\sk} is a polynomial-time algorithm if $\sigma_2 > 0$ is considered a fixed constant.

\begin{coro} Let $p = q = \frac{1}{n}\1_n $ (case 1 in \Cref{lem:bounded-norm}) and suppose $\sigma_2 > 0$ is a fixed constant. Then {\sk} finds an $\varepsilon$-approximate scaling in $\Ocal(n^3 + \log(\frac{n}{\varepsilon}))$ iterations. If $A > 0$ is positive (case 3 in \Cref{lem:bounded-norm}), then the complexity is further improved to $\Ocal(n + \log(\frac{n}{\varepsilon}))$.
\end{coro}

\begin{rem}
The result exhibits an undesired sharp complexity transition between $\Ocal(n)$ for positive matrices and $\Ocal(n^3)$ for general nonnegative matrices. This transition comes from the norm bound $D$ in \Cref{lem:bounded-norm} and is likely an artifact of the analysis. A smoothed measure of sparsity may overcome this weakness (e.g., \cite{he2026efficiency, he2026phase}).
\end{rem}

\paragraph{Matrix balancing and equilibration.}Given a real matrix $A \in
\mathbb{R}^{m \times n}$, the $\ell_{\omega}$ matrix equilibration
problem finds positive diagonal matrices $D_1, D_2$ such that
\[ | D_1 A D_2 |^{\omega} \1_n \approx p \quad \text{and} \quad | D_2 A^{\top} D_1
   |^{\omega} \1_m \approx q, \]
where $| A |$ denotes element-wise absolute value and $\omega \in (0,
\infty)$. This problem {\cite{fougner2018parameter,diamond2017stochastic}} can
be reduced to matrix scaling with data $(| A |^{\omega}, p, q)$, and our analysis follows immediately.

\begin{coro}
  For $\omega \in (0, 1)$, \Cref{alg:sinkhorn-altmin} applied to $(|A|^\omega, p, q)$ finds an $\varepsilon$-approximate $\ell_\omega$-equilibration of $A$ in
  $\mathcal{O} ( \tfrac{D^2}{\min \{ \sigma_2^4, s \sigma_2^2 \}}
  + \tfrac{1}{\sigma_2} \log ( \tfrac{1}{\varepsilon} ) )$
  iterations, where $\sigma_2 = \lambda_2 (I - P^{- 1 / 2} | A^{\star}
  |^{\omega} Q^{- 1} (| A^{\star} |^{\omega})^{\top} P^{- 1 / 2})$.
\end{coro}

Finally, since {\sk} can also be used for matrix balancing \cite{knight2008sinkhorn}, our results are similarly applicable.

\section{Nonasymptotic acceleration of matrix scaling} \label{sec:acc}

This section develops accelerated variants of  {\sk}. Given the tightness result in \Cref{thm: tight}, acceleration is
generally unachievable without modifying the algorithm. Hence, we resort to the semi-dual formulation {\cite{cuturi2018semidual}}.

\subsection{Semi-dual function}

Since {\sk} performs {\am} on $\varphi (u, v)$, it is natural to consider the partially minimized objective
\[ \zeta (u) \assign \min_v \varphi (u, v) = \varphi (u, - \log (q / A^{\top}
   \mathe^u))= \textstyle \sum_{j = 1}^n q_j \log \langle a_{[:, j]}, \mathe^u \rangle - \langle
  p, u \rangle - \sum_{j = 1}^n q_j \log q_j + \langle \1_n, q
  \rangle, \]
  where $a_{[:, j]}$ is the $j$-th column of $A$.
The function $\zeta$ is known as the semi-dual in the
literature {\cite{cuturi2018semidual}}. Its gradient can be
computed by a half  {\sk} iteration: define the nonlinear map $v (u) \assign - \log (q / A^{\top}
\mathe^u)$ and its Jacobian $\mathcal{J}_v$. We have
\[ \nabla \zeta (u) = \nabla_u \varphi (u, v (u)) + \nabla_v \varphi (u, v
   (u)) \mathcal{J}_v (u) = \nabla_u \varphi (u, v (u)),\]
   since the optimality condition implies $\nabla_v \varphi (u, v
   (u)) = 0$. Therefore, minimizing the residual $\|\nabla \varphi(u, v)\|$ reduces to reducing the gradient norm $\|\nabla \zeta(u)\|$. \Cref{lem:reduced-pot} shows that $\zeta$ has desirable properties.

\begin{lem}
  \label{lem:reduced-pot}The semi-dual function $\zeta$ satisfies the
  following properties:
  \begin{enumerate}[leftmargin=15pt]
    \item It is convex, with $L=\tfrac{1}{2}
  \| p \|_1$-Lipschitz gradient and $H = \tfrac{3}{2} \| p \|_1$-Lipschitz Hessian.
    
    \item If $\zeta (u) - \zeta (u^{\star}) \leq \tfrac{s \sigma_2^2}{1296}$, then $\zeta (u) - \zeta (u^{\star})
    \geq \tfrac{\sigma_2}{4} \| \Pi ( u - \us ) \|^2_P$
    and
    \[ \tfrac{4}{\sigma_2} \| \nabla \zeta (u) \|_{P^{- 1}}^2 \geq \zeta (u)
       - \zeta (u^{\star}) \geq \tfrac{1}{16} \| \nabla \zeta (u) \|_{P^{-
       1}}^2 . \]
\item If $\zeta (u) - \zeta (u^{\star}) \leq \min \{
   \tfrac{s \sigma_2^2}{1296}, \tfrac{\sigma_2^3 s^3}{36\| p \|_1^2}
    \}$, then $\tfrac{\sigma_2}{2}\leq\lambda_2 (P^{-1/2}\nabla^2 \zeta (u) P^{-1/2})$ and $\lambda_m (P^{-1/2}\nabla^2 \zeta (u) P^{-1/2}) \leq 2$.
  \end{enumerate}
where $\Pi = I - \frac{1}{\|p\|_1} \1_m p^\top$ is the orthogonal projection onto $p^\perp$ in $P$-norm.
\end{lem}

\begin{rem}
To our knowledge, the global and local smoothness (case 1 of \Cref{lem:reduced-pot}) of the semi-dual function have been established and leveraged in the literature for optimal transport \cite{cuturi2018semidual, genans2026fast, xu2026accelerating}. Recent work \cite{xu2026accelerating} has derived a PL constant for the semi-dual function for any given sublevel set. However, the explicit local PL constants (cases 2 and 3 of \Cref{lem:reduced-pot}) based on $\sigma_2$ are not yet available.
\end{rem}

Given that $\zeta$ is a smooth, convex function with a finite-sum structure, there are several techniques from optimization theory for making its gradient small.

\subsection{Global acceleration}
Finding an $\varepsilon$-approximate scaling corresponds to making $\| \nabla
\zeta (u) \| \leq \varepsilon$. For an $L$-smooth convex function, Nesterov's
regularization technique achieves this goal with an iteration complexity of $\mathcal{O} \big(
{\sqrt{\tfrac{L \| u^{\star}\|}{{\varepsilon}} }} \log ( \tfrac{L \|
u^{\star} \|}{\varepsilon} ) \big)$ {\cite{nesterov2013introductory}},
where the $\log$ factor can be removed using the recently developed performance estimation techniques
{\cite{lee2021geometric}}.

\begin{thm}
  \label{thm:global-acc} Under the same assumptions as \Cref{thm:global-conv}, there exists an algorithm $\Acal_1$ that outputs an
  $\varepsilon$-approximate scaling in
  \[ K_{\varepsilon} = \Big\lceil \tfrac{8 \sqrt{\| p \|_1 \| u^{\star}
     \|}}{\sqrt{\varepsilon}}  \Big\rceil \leq \Big\lceil \tfrac{8 m^{1
     / 4} \| p \|_1^{1/2} \sqrt{D}}{\sqrt{\varepsilon}}\Big\rceil \]
  iterations, where $(\us, \vs)$ is chosen to attain $D$ in \eqref{eqn:diam} and each iteration has the same order of cost as  {\sk}.
\end{thm}

The arithmetic complexity of \Cref{thm:global-acc} is $\mathcal{O} ( \tmop{nnz} (A) + \tfrac{\tmop{nnz}
(A) }{\sqrt{\varepsilon}} m^{1 / 4} {\| p \|_1^{1/2} }{ \sqrt{D}}
)$. For $p = q = \1_n$, this improves on the $\varepsilon^{- 2 / 3}$ dependence of the gradient-based method in \cite{allen2017much}. In this unit-margin setting, their inverse-margin weighted residual norm coincides with the Euclidean norm used here.
Finally, given the finite-sum structure of $\zeta$, applying stochastic
gradient descent with variance reduction (e.g., \texttt{Katyusha} {\cite{allen2018katyusha}})
further reduces the complexity to $\tilde{\mathcal{O}} ( \tmop{nnz} (A) + \tfrac{\tmop{nnz} (A)
}{\sqrt{\varepsilon}} m^{1 / 4} \| q \|_\infty^{1/2} \sqrt{D}
)$ in achieving $\mathbb{E} [\| \nabla \zeta (u) \|] \leq \varepsilon$.

\begin{thm} \label{thm:global-vr}
Under the same assumptions as \Cref{thm:global-conv}, there exists an algorithm $\Acal_2$ that outputs $\hat{u}$ such that $\mathbb{E} [\|
  \nabla \zeta (\hat{u}) \|] \leq \varepsilon$ with arithmetic complexity $\tilde{\mathcal{O}} (
  \tmop{nnz} (A) + \tfrac{\tmop{nnz} (A)}{\sqrt{\varepsilon}} m^{1 / 4}  \| q \|_\infty^{1/2} \sqrt{D})$ in expectation.
\end{thm}

\begin{rem}
 For $p = q=\1_n$, the expected arithmetic complexity is $\tilde{\mathcal{O}} (\tmop{nnz} (A) +
  \tfrac{\tmop{nnz} (A)}{\sqrt{\varepsilon}} n^{1 / 4}{
  \sqrt{D}} )$. For a dense $n \times n$ matrix, this becomes $\tilde{\mathcal{O}} (n^2 + n^{9 / 4} \sqrt{\tfrac{D}{\varepsilon}})$.
\end{rem}

\begin{rem}

  Consider optimal transport with $n \geq 2$, uniform probability margins $p = q = \tfrac{1}{n} \1_n$, a nonnegative cost matrix $C$, and target transport-cost accuracy $0 < \hat{\varepsilon} \leq \| C \|_{\infty}$. Using the Gibbs matrix $A = \mathe^{- C / \eta}$ and \cite{lin2019efficient}
  \[ \eta = \tfrac{\hat{\varepsilon}}{4 \log n}, \quad \varepsilon = \tfrac{\hat{\varepsilon}}{8 \sqrt{n} \| C \|_{\infty}} . \]
  We have $\tmop{nnz} (A) = n^2$, $\| q \|_\infty = n^{-1}$, and $D = \mathcal{O} (\tfrac{\| C \|_{\infty}}{\eta} + \log n)$. Applying \Cref{thm:global-vr} with matrix-scaling accuracy $\varepsilon$ gives
  \[ \mathbb{E} [\| \nabla \zeta (\hat{u}) \|_1] \leq \sqrt{n} \mathbb{E} [\| \nabla \zeta (\hat{u}) \|] \leq \tfrac{\hat{\varepsilon}}{8 \| C \|_{\infty}} . \]
  The scaled matrix has exact column margins. Rounding it to the prescribed margins costs $\mathcal{O} (n^2)$ operations and gives a feasible transport plan whose expected cost exceeds the optimal unregularized transport cost by at most \cite{altschuler2017near,lin2019efficient}
  \[ 2 \eta \log n + 4 \| C \|_{\infty} \mathbb{E} [\| \nabla \zeta (\hat{u}) \|_1] \leq \hat{\varepsilon} . \]
  The resulting expected arithmetic complexity is $\tilde{\mathcal{O}} \big( n^2 + n^{7 / 4} \sqrt{\tfrac{D}{\varepsilon}} \big) = \tilde{\mathcal{O}}
     \big( \tfrac{n^2 \| C \|_{\infty}}{\hat{\varepsilon}} \big) $.
  Thus, in this uniform-margin setting, the dependence on $n$, $\| C \|_{\infty}$, and $\hat{\varepsilon}$ matches the bounds in \cite{jambulapati2019direct,blanchet2018towards}, up to logarithmic factors.
  \end{rem}
While global acceleration is interesting from a worst-case perspective, it
is local behavior that determines how quickly the algorithm reaches a high-accuracy solution. The rest of this section explores the local
behavior of  {\sk} and introduces two acceleration mechanisms based
on the semi-dual formulation: 
one adopts Nesterov's acceleration and improves
the complexity from $\mathcal{O} ( \tfrac{1}{\sigma_2} \log (
\tfrac{1}{\varepsilon} ) )$ to $\mathcal{O} (
\tfrac{1}{\sqrt{\sigma_2}} \log ( \tfrac{1}{\varepsilon} )
)$; the other notices that locally,  {\sk}
corresponds to a {\tmem{suboptimal}} diagonal preconditioner. A better preconditioner improves the local contraction.

\subsection{Local suboptimality of Sinkhorn-Knopp} \label{sec:local-subopt}

Consider a step of  {\sk}: $(u, v) \rightarrow (u^+, v)$ with $\nabla_v \varphi
(u, v) = 0$. By our previous analysis, locally, we have
\begin{equation}
	u^+ = u - P^{- 1} \nabla_u \varphi (u, v) = u - P^{- 1} \nabla \zeta (u) .
   \label{eqn:local-pgd}
\end{equation}
Therefore, the local convergence rate of {\sk} can be reproduced by running
preconditioned gradient descent on $\zeta$ with preconditioner $P^{- 1}$ (or
$Q^{- 1}$ if $v$ is kept in the semi-dual formulation). Since a preconditioner can be viewed as a matrix stepsize, a
natural question is whether this stepsize is optimal. To simplify the notation, for now we assume $p = q = \1_m$ and $P = I$, and \eqref{eqn:local-pgd}
simplifies to vanilla gradient descent with stepsize $1$. With this simplification, the local linear rate of preconditioned gradient descent is dictated by the ratio between smoothness and strong convexity over $\1_m^{\perp}$, 
the orthogonal complement of $\tmop{span} \{ \1_m \}$:
\[ \sigma_2 I ~\preceq_{\Pi}~ \nabla^2 \zeta (u^{\star}) = I - R 
   R^{\top}~ \preceq_\Pi ~\sigma_m  I, \]
where $\preceq_{\Pi}$ denotes the semidefinite order over $\1_m^{\perp}$ and $\sigma_m \assign \lambda_{m} (I - R
 R^{\top}) \leq 1$.
The effective local condition number is 
$\tfrac{\sigma_m}{\sigma_2} \leq \tfrac{1}{\sigma_2}$, and the $(1 - \sigma_2)^2 $
contraction of {\sk} (two half-iterations) matches this bound.
 However, this perspective also shows that {\sk} is locally
suboptimal: for an $L$-smooth $\mu$-strongly
convex problem, the minimax optimal stepsize \cite{taylor2018exact} $\tfrac{2}{L + \mu} >
\frac{1}{L}$ produces a better contraction
\[ (\tfrac{L - \mu}{L + \mu})^2 = (1 -
   \tfrac{2}{\kappa + 1})^2 < (1 - \tfrac{1}{\kappa})^2 . \]
Plugging in $L = 1$ and $\mu = \sigma_2$, we see that the stepsize
$\tfrac{2}{\sigma_2 + 1}$ provides a local contraction of $(\tfrac{1 -
\sigma_2}{1 + \sigma_2})^2 < (1 - \sigma_2)^2$. If $\sigma_2 \rightarrow 1$, this stepsize improves on the contraction of {\sk} by nearly a factor of $4$. Furthermore, we have been using $L = 1$, an upper bound on $\sigma_m$; in contrast, the
true condition number $\tfrac{\sigma_m}{\sigma_2}$ can be smaller than
$\tfrac{1}{\sigma_2}$, making this change even more impactful. Indeed, when $R$ has full row rank, $R R^{\top}$
is positive definite, and
\[ \sigma_m = \lambda_{m} (I - R R^{\top}) < 1. \]
Together with the previous analysis, the stepsize
$\tfrac{2}{\sigma_2 + \sigma_m}$ further improves the contraction to
$(\tfrac{\sigma_m - \sigma_2}{\sigma_m + \sigma_2})^2< (1 - {\sigma_2})^2$.

\begin{rem}
Our analysis reveals an asymmetry between the $u$ and $v$
blocks: to ensure $R$ is full row rank, it is necessary that $m \leq n$. In
other words, $\sigma_m <1$ happens only if one uses the semi-dual objective on the block with a smaller dimension. We are not aware of this viewpoint being exploited in existing {\sk} analyses.
\end{rem}

As a summary, the local perspective reveals the intrinsic suboptimality of {\sk}.
This observation also carries over to the local behavior of other two-block alternating minimization algorithms.
Using a slightly larger stepsize provably improves
the local contraction. However, there is one challenge left: the above stepsizes rely on unknown information such as $\sigma_2$ and $\sigma_m$. It is computationally expensive to obtain accurate estimates of these quantities. To use this strategy, we need 
an algorithm that automatically finds the locally best stepsize. This goal is achieved by the online scaled gradient
method (\texttt{OSGM}) \cite{gao2025gradient}.

\subsection{Local online acceleration}

Consider optimizing $\zeta$ with gradient descent preconditioned (scaled) by a diagonal matrix
$\mathcal{D} (w)$ (vector $w$)
\begin{equation} \label{eqn:scal-gd}
	u^+ = u -\mathcal{D} (w) \nabla \zeta (u) . 
\end{equation}
The descent lemma in $P$-norm  motivates the definition of relative progress of the preconditioner vector $w$ with respect to the scaled gradient norm $\|
\nabla \zeta (u) \|^2_{P^{- 1}}$: $h_u (w) \assign \tfrac{\zeta (u -\mathcal{D} (w) \nabla \zeta (u)) - \zeta
   (u)}{\| \nabla \zeta (u) \|^2_{P^{- 1}}}$.
The function $h_u$ inherits properties like convexity and smoothness from $\zeta$.
\begin{lem}
  \label{lem:hu-property} The function $h_u$ has the following properties:
  \begin{enumerate}[leftmargin=15pt]
  	\item It is convex and $L=\frac{1}{2}\|p\|_1$-smooth in $P$-norm,
  	\item $h_u ( P^{-1} \1_m ) \leq - \tfrac{1}{4}$ for all $u$ such that $\zeta (u) - \zeta (u^{\star}) \leq \tfrac{\sigma_2^2 s}{1296}$. 
  \end{enumerate}
\end{lem}

\begin{algorithm}[h]
{\textbf{input} $(A, p, q)$ with $L = \frac{1}{2}\|p\|_1, H = \frac{3}{2}\|p\|_1$ (as defined in \Cref{lem:reduced-pot}), initial $u^1$ and $w^1$}

\For{$k = 1, 2, \dots$}{
$w^{k+1} = w^k - \frac{1}{L} P^{-1}\nabla h_{u^{k}}(w^k)$\\
$u^{k+1/2} = u^k - \Dcal(w^{k+1}) \nabla \zeta(u^k)$\\
Choose $u^{k+1}$ that yields the smaller potential value $\zeta$ between $u^{k+1/2}$ and $u^k$.
}
\caption{Matrix scaling with online scaling ({\osms})\label{alg:osms}}
\end{algorithm}

Online acceleration (\Cref{alg:osms}) alternates between \textbf{1)}  (hyper-)gradient descent on $w$ (with gradient $\nabla h_u(w)$) to learn $\hat{w}$ and \textbf{2)} a preconditioned gradient update \eqref{eqn:scal-gd} on $u$ (with gradient $\nabla \zeta(u)$). Given any preconditioner vector $\hat{w}$, define the learning potential 
\begin{align}
  \Omega_{\hat{w}} (u, w) \assign & \log (\zeta (u) - \zeta (u^{\star})) + \tfrac{L
  \sigma_2}{8} \| w - \hat{w} \|^2_P . \nonumber
\end{align}

The following result characterizes the local behavior of \Cref{alg:osms}.

\begin{lem}
  \label{lem:pot}Under the same assumptions as \Cref{thm:global-conv}, for any fixed scaling vector $\hat{w} \in \Rbb^m$,  we have  $\Omega_{\hat{w}} (u^+, w^+) \leq \Omega_{\hat{w}} (u, w) +
  \tfrac{\sigma_2}{4} h_u (\hat{w})$ for all $u$ such that $\zeta (u) - \zeta (u^{\star}) \leq \tfrac{s\sigma_2^2}{1296} $.
\end{lem}

Telescoping \Cref{lem:pot} implies the relation $\Omega_{\hat{w}} (u^{K + 1}, w^{K + 1})
\leq \Omega_{\hat{w}} (u^1, w^1) +  \tfrac{\sigma_2}{4}\sum_{k = 1}^K h_{u^k} (\hat{w})$ holds
for any $\hat{w}$. Given the freedom to choose $\hat{w}$, taking $\hat{w} =
P^{-1} \1_m$ gives $h_u ( P^{-1} \1_m ) \leq - \tfrac{1}{4}$ for all local $u$ (by \Cref{lem:hu-property}), which
gives $\Omega_{\hat{w}} (u^{K + 1}, w^{K + 1}) \leq \Omega_{\hat{w}} (u^1, w^1) -
\tfrac{\sigma_2}{16} K$ and a final complexity of
\[ K_{\varepsilon} \assign \big\lceil \tfrac{16}{\sigma_2} \tfrac{L
   \sigma_2}{8} \| w^1 - P^{-1} \1_m \|^2_P + \tfrac{16}{\sigma_2}
   \log ( \tfrac{1}{\varepsilon} ) \big\rceil = \big\lceil 2 L
   \| w^1 - P^{-1} \1_m \|^2_P + \tfrac{16}{\sigma_2} \log (
   \tfrac{1}{\varepsilon} ) \big\rceil . \]
Alternatively, we can take $\hat{w} = w^{\star}$ that minimizes $\sum_{k =
1}^K h_{u^k} (\hat{w})$. Define $\sigma^{\star}_2 \assign -
\tfrac{\sigma_2}{4} \max_u h_u (w^{\star})$. Then $\Omega_{w^{\star}} (u^{K +
1}, w^{K + 1}) \leq \Omega_{w^{\star}} (u^1, w^1) - \sigma_2^{\star} K$ implies
a complexity of
\[ K_{\varepsilon} \assign \lceil \tfrac{L \sigma_2}{8 \sigma_2^{\star}}
   \| w^1 - w^{\star} \|^2_P + \tfrac{1}{\sigma^{\star}_2} \log (
   \tfrac{1}{\varepsilon} ) \rceil, \]
which can be faster than $\mathcal{O} ( \tfrac{1}{\sigma_2} \log (
\tfrac{1}{\varepsilon} ) )$ of {\sk} if $\sigma_2 <
\sigma^{\star}_2$. Finally, the above local arguments are automatically
globalized since \texttt{OSGM} has an $\mathcal{O} ( \tfrac{D^2}{K}
)$ global convergence rate on smooth convex functions \cite{gao2025gradient}. 
\begin{thm}[Informal] 
\Cref{alg:osms} outputs an $\varepsilon$-approximate scaling in  $\Ocal(\frac{1}{\sigma_2^\star}\log (
\tfrac{1}{\varepsilon} ))$ iterations.
\end{thm}
We see that while online acceleration can speed up matrix scaling, its effect is problem-dependent. On the other hand, extrapolation-based schemes often yield a strict acceleration effect. The next section explores Nesterov's accelerated gradient method to accelerate matrix scaling locally.

\subsection{Local Nesterov's acceleration}

A few attempts have been made in the literature to accelerate {\sk}. On the one hand, global accelerated variants of {\sk} often admit elegant nonasymptotic analysis \cite{xu2026accelerating,tang2024accelerating}. On the other hand, local acceleration of {\sk} mostly works by interpreting it as a fixed-point iteration and employing overrelaxation {\cite{lehmann2022note,thibault2021overrelaxed,peyre2019quantum}}. However, existing arguments based on fixed-point iteration are asymptotic. In this section, we instead make use of the preconditioned gradient descent perspective, whose accelerated variant is preconditioned accelerated gradient descent (\Cref{alg:accmb}, \pagd).
\begin{algorithm}[h]
{\textbf{input} $(A, p, q)$, initial point $u^1=w^1$ }

\For{$k = 1, 2, \dots$}{
$y^{k} = u^k + \tfrac{\sqrt{\sigma_2}}{\sqrt{\sigma_2} + 2} (w^k - u^k)$\\
$u^{k+1/2} = y^k - \frac{1}{2}P^{- 1} \nabla \zeta (y^k)$\\
$w^{k+1} = ( 1 - \frac{1}{2}\sqrt{\sigma_2} ) w^k + \frac{1}{2}\sqrt{\sigma_2} ( y^k
  - \tfrac{2}{\sigma_2} P^{- 1} \nabla \zeta (y^k) )$\\
Let $u^{k+1}$ be the better between $u^{k+1/2}$ and $u^k$.
}
\caption{Matrix scaling with accelerated preconditioned gradient descent ({\pagd})\label{alg:accmb}}
\end{algorithm}
Our analysis adopts the following potential function \cite{d2021acceleration}
\[ f (u, w) \assign \zeta (u) + \tfrac{\sigma_2}{4} \| \Pi ( w - \us
   ) \|^2_P, \]
   where recall that $\Pi = I - \frac{1}{\|p\|_1} \1_m p^\top$. By \Cref{lem:reduced-pot}, locally, $\zeta$ is 2-smooth and $\tfrac{\sigma_2}{2}$-strongly convex over
$p^\perp$ in $P$-norm. Then, a standard potential function-based
analysis {\cite{d2021acceleration}} establishes the desired accelerated rate.

\begin{lem} \label{lem:local-contraction}
Let  $(u, w), (u^+, w^+)$ be consecutive iterates generated by \Cref{alg:accmb}. If $\zeta (u^1) - \zeta ( \us ) \leq \tfrac{1}{2} \min \{\tfrac{\sigma_2^2 s}{1296}, \tfrac{\sigma_2^3 s^3}{36\| p \|_1^2}
    \} \cdot \min \{ \tfrac{s \sigma_2}{\| p \|_1}, 1 \}$, then $f (u^+, w^+) - \zeta ( \us ) \leq ( 1 - \tfrac{1}{2}
     \sqrt{\sigma_2} ) [ f (u, w) - \zeta ( \us )
     ]$.
\end{lem}

\begin{thm} If $u^1 = w^1$ satisfies the condition from \Cref{lem:local-contraction}, then \Cref{alg:accmb} outputs an $\varepsilon$-approximate scaling in  $\Ocal( \frac{2}{\sqrt{\sigma_2}}\log (
\tfrac{1}{\varepsilon} ))$ iterations.
\end{thm}

\section{Numerical experiments} \label{sec:exp}

This section conducts numerical experiments to validate our findings.

\paragraph{Experiment setup.}
We compare {\sk}, gradient descent (\texttt{GD}) on $\zeta$, {\osms}
(\Cref{alg:osms}), and {\pagd} (\Cref{alg:accmb}). Both {\osms} and {\pagd}
warm-start with {\sk} until $\| \mathe^{U} A \mathe^{- v} - p \| \leq 10^{- 3}$;
{\osms} initializes its preconditioner at $P_1 = P^{- 1}$ (so its first step is a
{\sk} step). 

\paragraph{Dataset.} We use entropy-regularized optimal transport instances with Gibbs kernel $A = \mathe^{- C / \eta}$, $\eta = 2 \times 10^{- 3}$, with the following data: (1) \texttt{random}: support points are uniform in $[0, 1]^2$, $C$ is the squared Euclidean cost, and $p, q \sim \mathcal{U}[0,1]$; (2) \texttt{MNIST}: $p, q$ are normalized digit images with
$m = n = 784$.

\subsection{Local suboptimality of Sinkhorn-Knopp}

This section validates our finding that {\sk} is locally suboptimal (using a fixed stepsize of 1.0). In \Cref{fig:local}, we compare the performance of {\sk} to \texttt{GD} with two fixed stepsizes $\frac{2}{1 + \sigma_2}$ and $\frac{2}{\sigma_2 + \sigma_m}$ from \Cref{sec:local-subopt}.

\begin{figure}[h]
  \centering
  \includegraphics[width=0.95\textwidth]{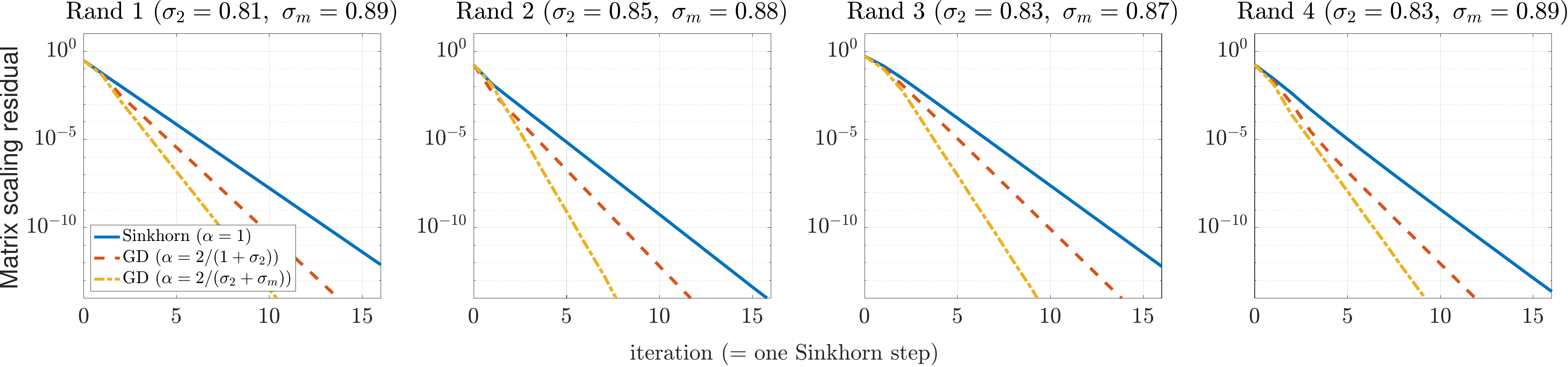}
  \caption{Local suboptimality of {\sk} ($m = 3$, $n = 300$): the larger
  minimax ($\alpha = 2 / (1 + \sigma_2)$) and optimal
  ($\alpha = 2 / (\sigma_2 + \sigma_m)$) stepsizes contract faster than {\sk}
  ($\alpha = 1$).}
  \label{fig:local}
\end{figure}

As our theory predicts, using a larger step size yields better local contraction than {\sk}.

\subsection{Accelerated variants of Sinkhorn-Knopp}

\Cref{fig:acc-rand,fig:acc-mnist} compare the accelerated variants on eight
random and eight \texttt{MNIST} instances. Both {\pagd} and {\osms} beat {\sk} by several
orders of magnitude within the budget. 
\begin{figure}[h]
  \centering
  \includegraphics[width=0.95\textwidth]{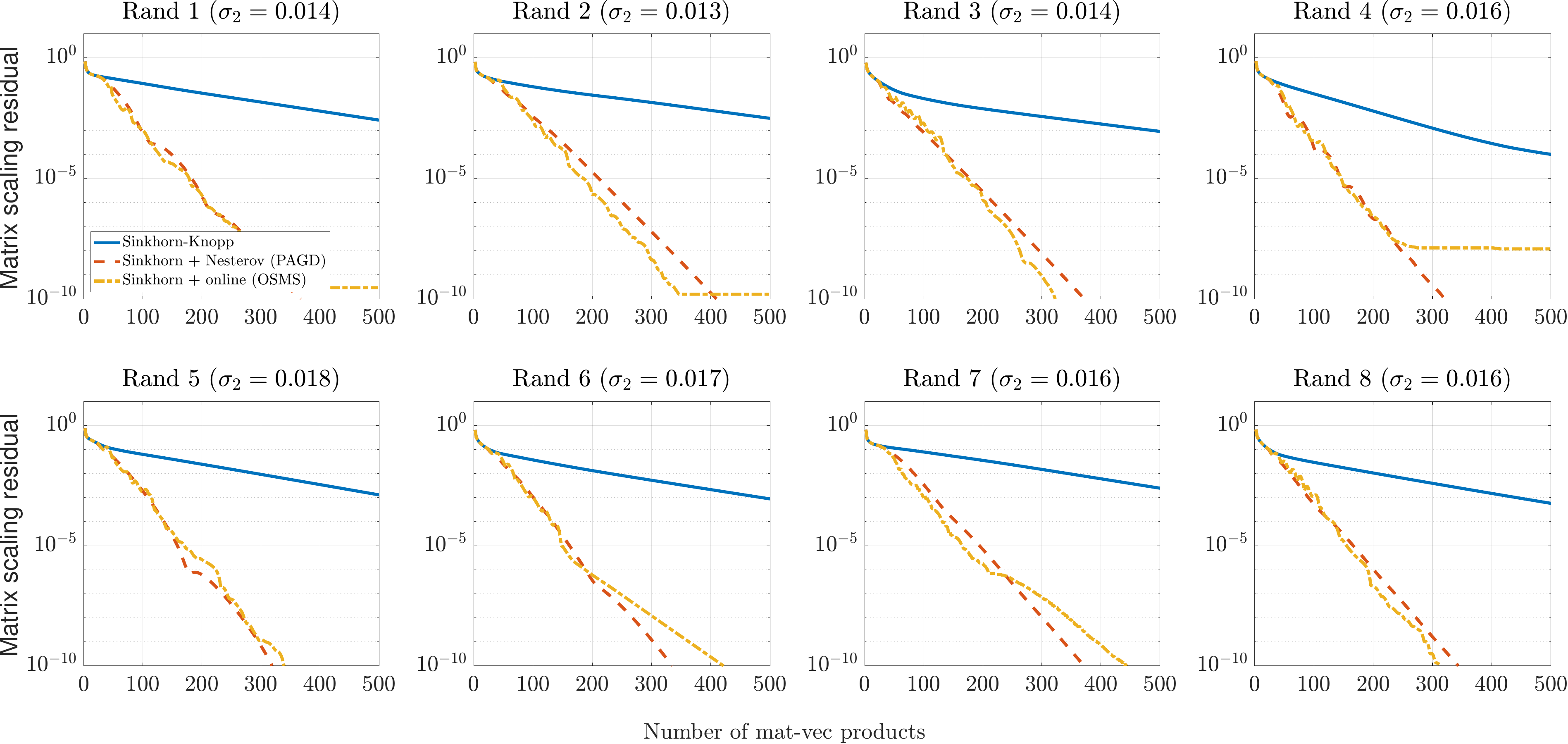}
  \caption{Accelerated variants of {\sk} on eight random entropic optimal
  transport instances ($m = n = 200$): residual versus
  number of matrix-vector products.}
  \label{fig:acc-rand}
\end{figure}

\begin{figure}[h]
  \centering
  \includegraphics[width=0.95\textwidth]{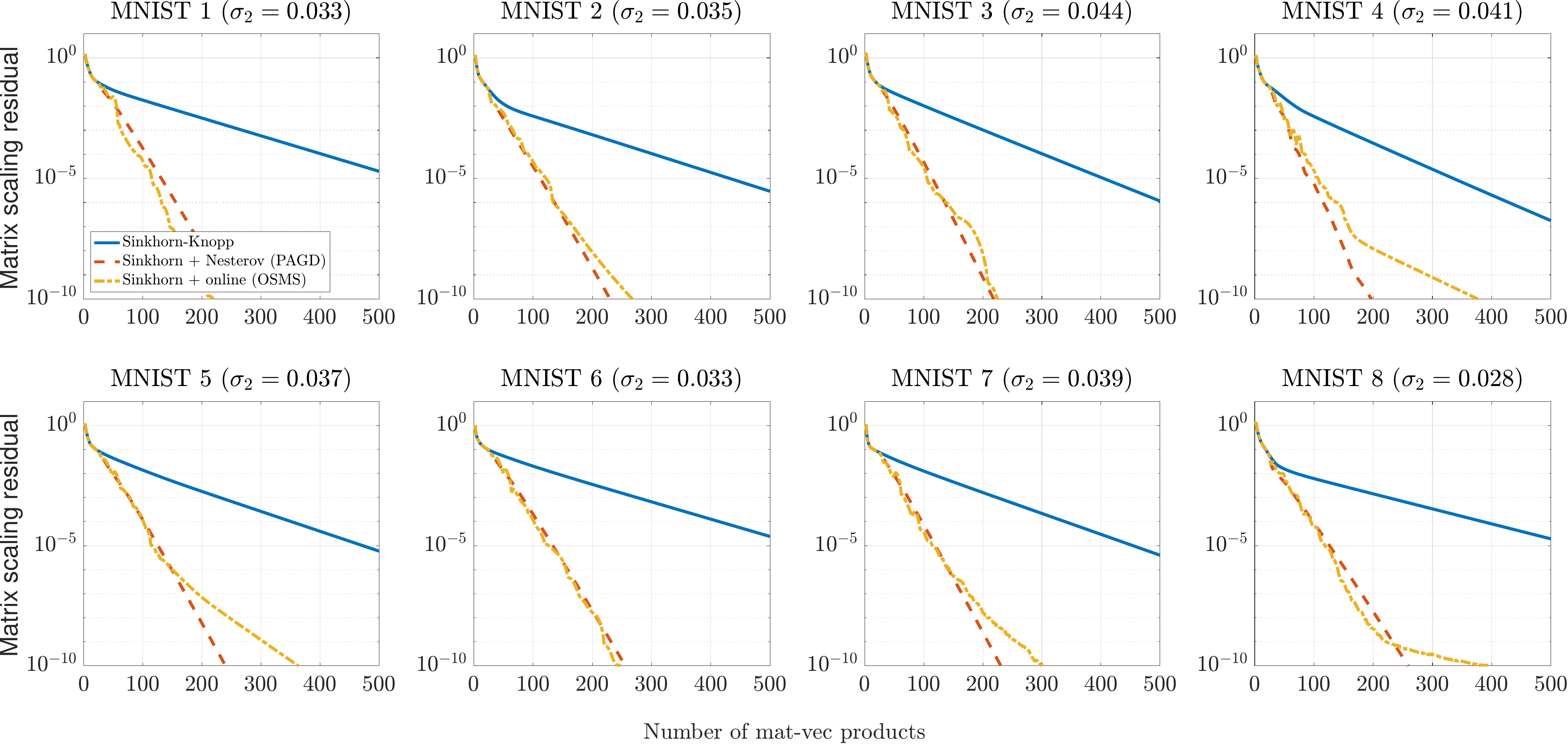}
  \caption{Accelerated variants of {\sk} on eight \texttt{MNIST} entropic optimal
  transport instances ($n = 784$, $\eta = 2 \times 10^{- 3}$).}
  \label{fig:acc-mnist}
\end{figure}

\section{Conclusion}
This paper investigates the nonasymptotic local convergence of {\sk} and provides new insights into its behavior. We also develop nonasymptotic accelerated variants through the semi-dual formulation.
Our techniques extend to related problems, such as unbalanced optimal transport \cite{genans2026fast}, and our analysis template applies to general two-block alternating minimization algorithms \cite{gao2026nonasymptotic}.

\newpage

\renewcommand \thepart{}
\renewcommand \partname{}

\bibliography{ref.bib}
\bibliographystyle{plain}
\doparttoc
\faketableofcontents
\part{}

\newpage
\appendix

\addcontentsline{toc}{section}{Appendix}
\part{Appendix} 
\parttoc

\newpage
\section{Proofs of results in Section \ref{sec:global}}

\subsection{Auxiliary results}

\begin{lem}
  \label{lem:bounded-norm-dsto}Suppose \ref{A1} holds and that $p =
  \tfrac{1}{m} \1_m$, $q = \tfrac{1}{n} \1_n$. Then there exists a solution
  $( \us, \vs )$ such that
  \[ \| ( \us, \vs ) \|_{\infty} \leq \tfrac{1}{2}
     | \log ( \tfrac{\| (p, q) \|_{\infty}}{\| A \|_{- \infty}}
     ) | + \tmop{lcm} (m, n) \log ( \tfrac{\| (p,
     q) \|_{\infty} \| A \|_1}{\| p \|_1 \| A \|_{- \infty}} ), \]
  where $\tmop{lcm} (m, n)$ denotes the least common multiple of $m$ and $n$.
\end{lem}

\begin{proof}
According to Theorem 5.1 of
{\cite{kalantari2008complexity}}, there exist $( \us, \vs )$ such
that
\[ \us_i = \tfrac{1}{2} \log ( \tfrac{\| (p, q) \|_{\infty}}{\| A \|_{-
   \infty}} ) + t \quad \text{and} \quad \vs_j = \tfrac{1}{2} \log
   ( \tfrac{\| (p, q) \|_{\infty}}{\| A \|_{- \infty}} ) - t, \]
where
\[ | t | \leq \tfrac{1}{\delta_{m, n}} \big| - \| p \|_1 \log (
   \tfrac{\| A \|_1}{\| p \|_1} ) - \| p \|_1 \log \big( \tfrac{\| (p,
   q) \|_{\infty}}{\| A \|_{- \infty}} \big) \big| = \tfrac{1}{\delta_{m,
   n}} \| p \|_1 \log \big( \tfrac{\| (p, q) \|_{\infty} \| A \|_1}{\| p \|_1
   \| A \|_{- \infty}} \big), \]
and $\delta_{m, n} = \tfrac{1}{\tmop{lcm} (m, n)}$, where $\tmop{lcm}(m, n)$ denotes the least common multiple of $m$ and $n$. Therefore, with $\|p\|_1= 1$, we have
\[ | \us_i | \leq \tfrac{1}{2} \big| \log ( \tfrac{\| (p, q)
   \|_{\infty}}{\| A \|_{- \infty}} ) \big| + \tmop{lcm} (m, n)\log ( \tfrac{\| (p, q) \|_{\infty} \| A \|_1}{ \| A
   \|_{- \infty}} ), \]
and the same result holds for $\vs$.
\end{proof}

\begin{lem}
  \label{lem:bounded-norm-pos}Suppose \ref{A1} holds and that $m = n, A >
  0$ and $\|p\|_1= 1$. Then there exists a solution $( \us, \vs )$
  such that $\| ( \us, \vs ) \|_{\infty} \leq \log
  ( \tfrac{n}{\| A \|_{- \infty}  \| (p, q) \|_{- \infty}^2} )$.
\end{lem}

\begin{proof}
The proof is immediate by taking $C_{i j} = \mathe^{- a_{i j}}$ and adopting
$v \leftarrow - v$ in {\cite{lin2019efficient}}.
\end{proof}

\subsection{Proof of Lemma \ref{lem:bounded-norm}}

The first two bullet points follow from \Cref{lem:bounded-norm-dsto} with
$\tmop{lcm} (m, n) \leq m n$ and $\tmop{lcm} (n, n) = n$. The last bullet
point follows from \Cref{lem:bounded-norm-pos}.

\subsection{Proof of Theorem \ref{thm:global}}

We start by showing that the Sinkhorn iterates satisfy $\| u^{k + 1} -
\us \|_{\infty} \leq \| u^1 - \us \|_{\infty}$. According to Lemma 2 of {\cite{qu2025sinkhorn}},
consider any $i \in [m]$ and $\gamma, \eta$ such that
\[ \gamma \mathe^{u^{\star}_i} \leq \mathe^{u^k_i} \leq \eta
   \mathe^{u^{\star}_i}, \]
we always have $\gamma \mathe^{u^{\star}_i} \leq \mathe^{u^{k + 1}_i} \leq
\eta \mathe^{u^{\star}_i}$. Taking $\gamma = \mathe^{u^{\star}_i - u^1_i}$ and
$\eta = \mathe^{u^1_i - u^{\star}_i}$, we have, inductively, that
\[ \mathe^{u^{\star}_i - u^1_i} \mathe^{u^{\star}_i} \leq \mathe^{u_i^k} \leq
   \mathe^{u^1_i - u^{\star}_i} \mathe^{u^{\star}_i}, \]
which implies $\mathe^{u^{\star}_i - u^1_i} \leq \mathe^{u_i^k - u^{\star}_i}
\leq \mathe^{u^1_i - u^{\star}_i}$ and that $| u_i^k - u^{\star}_i | \leq |
u^1_i - u^{\star}_i |$, giving $\| u^k - \us \|_{\infty} \leq
\| u^1 - \us \|_{\infty}$. The same argument holds for $v^k$.
Next, we deduce that
\begin{align}
  \varphi (u^k, v^k) - \varphi ( \us, \vs ) \leq{} & \langle
  \nabla \varphi (u^k, v^k), ( u^k - \us, v^k - \vs ) \rangle
  \label{eqn:bounded-norm-1}\\
  \leq{} & \| \nabla \varphi (u^k, v^k) \|_1 \| ( u^k - \us, v^k - \vs
  ) \|_{\infty} \label{eqn:bounded-norm-2}\\
  \leq{} & \| \nabla \varphi (u^k, v^k) \|_1 \| ( u^1 - \us, v^1 - \vs
  ) \|_{\infty} \label{eqn:bounded-norm-3}\\
  ={} & \| \nabla \varphi (u^k, v^k) \|_1 \| ( \us, \vs )
  \|_{\infty} \label{eqn:bounded-norm-4}\\
  \leq{} & D \| \nabla \varphi (u^k, v^k) \|_1 \nonumber,
\end{align}
where \eqref{eqn:bounded-norm-1} uses convexity of $\varphi$; \eqref{eqn:bounded-norm-2} uses H\"{o}lder's inequality $\langle a, b \rangle \leq \|a\|_1 \|b\|_\infty$; \eqref{eqn:bounded-norm-3} uses the previously derived relation $\| u^k - \us \|_{\infty} \leq
\| u^1 - \us \|_{\infty}$ and $\| v^k - \vs \|_{\infty} \leq
\| v^1 - \vs \|_{\infty}$; the relation \eqref{eqn:bounded-norm-4} uses the initialization $(u^1, v^1) = ( \0_m, \0_n )$. Finally, the proof of sublinear convergence is similar to that in {\cite{dvurechensky2018computational}}. Denote $\delta_k
\assign \varphi (u^k, v^k) - \varphi ( \us, \vs )$. According to
Theorem 1 of {\cite{dvurechensky2018computational}}, given
$\varphi (u^k, v^k) - \varphi ( \us, \vs ) \leq D \| \nabla
     \varphi (u^k, v^k) \|_1$, we have
\[ \delta_{k + 1} \leq \delta_k - \tfrac{\delta_k^2}{2 \|p\|_1 D^2}, \]
which implies $\tfrac{1}{\delta_{k + 1}}  -  \tfrac{1}{\delta_k} = \tfrac{\delta_{k }-\delta_{k + 1}}{\delta_{k}\delta_{k + 1}} \geq \tfrac{1}{2 \|p\|_1 D^2}
   \tfrac{\delta_k}{\delta_{k + 1}} \geq  \tfrac{1}{2
  \|p\|_1 D^2}$ since {\am} enforces $\delta_{k + 1} \leq \delta_k$. Hence
\[ \tfrac{1}{\varphi (u^{K + 1}, v^{K + 1}) - \varphi ( \us, \vs
   )} = \tfrac{1}{\delta_{K + 1}} \geq \tfrac{1}{\delta_1} + \tfrac{K}{2 \|p\|_1
   D^2} \]
and we arrive at $\varphi (u^{K + 1}, v^{K + 1}) - \varphi ( \us, \vs ) \leq
\tfrac{2 \|p\|_1 D^2}{K}$. This completes the proof.

\subsection{Proof of Proposition \ref{prop-equivalence}}
This equation can be directly verified. Let $\Delta u \assign u - u^{\star}$
and $\Delta v \assign v - v^{\star}$ and consider
\begin{align}
 \varphi (u, v) - \varphi (u^{\star}, v^{\star})  ={} & \langle \mathe^u, A \mathe^{- v} \rangle - \langle p, u \rangle +
  \langle q, v \rangle - (\langle \mathe^{u^{\star}}, A \mathe^{- v^{\star}}
  \rangle - \langle p, u^{\star} \rangle + \langle q, v^{\star} \rangle)
  \nonumber\\
  ={} & \langle \mathe^u, A \mathe^{- v} \rangle - \langle \mathe^{u^{\star}}, A
  \mathe^{- v^{\star}} \rangle - \langle p, u - u^{\star} \rangle + \langle q,
  v - v^{\star} \rangle . \nonumber
\end{align}
Using $\langle \mathe^u, A \mathe^{- v} \rangle = \langle \mathe^{u^{\star}}
\circ \mathe^{u - u^{\star}}, A (\mathe^{v^{\star} - v} \circ \mathe^{-
v^{\star}}) \rangle = \langle \mathe^{u - u^{\star}}, \mathe^{U^{\star}} A
\mathe^{- V^{\star}} \mathe^{v^{\star} - v} \rangle$, we deduce that
\begin{align}
  \langle \mathe^u, A \mathe^{- v} \rangle - \langle \mathe^{u^{\star}}, A
  \mathe^{- v^{\star}} \rangle 
  ={} & \langle \mathe^{u - u^{\star}}, \mathe^{U^{\star}} A \mathe^{-
  V^{\star}} \mathe^{v^{\star} - v} \rangle - \langle \mathe^{u^{\star}}, A
  \mathe^{- v^{\star}} \rangle \nonumber\\
  ={} & \textstyle \sum_{i, j} a_{i j} \mathe^{u^{\star}_i - v^{\star}_j} \mathe^{\Delta
  u_i - \Delta v_j} - \sum_{i, j} a_{i j} \mathe^{u^{\star}_i - v^{\star}_j}
  \nonumber\\
  ={} & \textstyle \sum_{i, j} a_{i j} \mathe^{u^{\star}_i - v^{\star}_j} [\mathe^{\Delta
  u_i - \Delta v_j} - 1] . \nonumber
\end{align}

Since $p_i = \sum_j a_{i j} \mathe^{u^{\star}_i - v^{\star}_j}$ and $q_j =
\sum_i a_{i j} \mathe^{u^{\star}_i - v^{\star}_j}$, we have
\begin{align}
  - \langle p, u - u^{\star} \rangle + \langle q, v - v^{\star} \rangle = & \textstyle -
  \sum_{i, j} a_{i j} \mathe^{u^{\star}_i - v^{\star}_j} \Delta u_i + \sum_{i,
  j} a_{i j} \mathe^{u^{\star}_i - v^{\star}_j} \Delta v_j \nonumber\\
  = & \textstyle - \sum_{i, j} a_{i j} \mathe^{u^{\star}_i - v^{\star}_j} (\Delta u_i -
  \Delta v_j) . \nonumber
\end{align}
Plugging in $a_{i j} \mathe^{u^{\star}_i - v^{\star}_j} = a^{\star}_{i j}$
gives
\begin{align}
\varphi (u, v) - \varphi (u^{\star}, v^{\star})
  ={} & \textstyle\sum_{i, j} a_{i j} \mathe^{u^{\star}_i - v^{\star}_j} [\mathe^{\Delta
  u_i - \Delta v_j} - 1] - \sum_{i, j} a_{i j} \mathe^{u^{\star}_i -
  v^{\star}_j} (\Delta u_i - \Delta v_j) \nonumber\\
  ={} & \textstyle\sum_{i, j} a_{i j} \mathe^{u^{\star}_i - v^{\star}_j} [\mathe^{\Delta
  u_i - \Delta v_j} - 1 - (\Delta u_i - \Delta v_j)] \nonumber\\
  ={} & \textstyle\sum_{i, j} a_{i j} \mathe^{u^{\star}_i - v^{\star}_j}
  (\mathe^{\Delta_{i j}} - 1 - \Delta_{i j}) \nonumber\\
  ={} & \textstyle\sum_{i, j} a_{i j}^{\star} \phi (\Delta_{i j}), \nonumber
\end{align}
and this completes the proof.

\subsection{Proof of Lemma \ref{lem:phi}}
Given $\phi (\delta) = \mathe^{\delta} - \delta - 1$, we have $\phi' (\delta)
= \mathe^{\delta} - 1$ and $\phi'' (\delta) = \mathe^{\delta}$. Hence $\phi'
(\delta) - \delta = \phi (\delta)$ and
\begin{align}
  | \phi' (\delta) | ={} & | \phi'' (\delta) - 1 | = | \mathe^{\delta} - 1 | .
  \nonumber
\end{align}
If $\delta > 0$, by $\mathe^{\delta} - \delta - 1 \geq \tfrac{\delta^2}{2}$,
we have $| \delta | \leq \sqrt{2 \phi (\delta)}$ and that
\[ | \mathe^{\delta} - 1 | \leq{} | \mathe^{\delta} - \delta - 1 | + | \delta |
   \leq \phi (\delta) + \sqrt{2 \phi (\delta)} . \]
If $\delta < 0$, we deduce that
\begin{align}
  | \mathe^{\delta} - 1 | ={} & 1 - \mathe^{\delta} \leq \sqrt{2
  (\mathe^{\delta} - \delta - 1)} = \sqrt{2 \phi (\delta)} \leq \phi (\delta)
  + \sqrt{2 \phi (\delta)}, \nonumber
\end{align}
where the first inequality follows from defining
\[ h (\delta) \assign (1 - \mathe^{\delta})^2 - 2 (\mathe^{\delta} - \delta -
   1) = \mathe^{2 \delta} - 4 \mathe^{\delta} + 2 \delta + 3 \]
and noticing that $h (0) = 0$ and that $h' (\delta) = 2 \mathe^{2 \delta} - 4
\mathe^{\delta} + 2 = 2 (\mathe^{\delta} - 1)^2 \geq 0$, giving $h (\delta)
\leq 0$ for $\delta \leq 0$. Combining these two cases completes the proof.
\subsection{Proof of Lemma \ref{lem:contraction}}

For brevity, in the proof we will denote $R \assign A^{\star}$ and hence
\[ \nabla^2 \varphi (u^{\star}, v^{\star}) = \Big(\begin{smallmatrix}
     P & - R\\
     - R^{\top} & Q
   \end{smallmatrix}\Big) . \]
Since $(u, v)$ is generated by {\am}, assuming $\nabla_v \lambda (u, v) = 0$, we have 
\begin{equation} \label{eqn:contraction-1}
v - v^{\star} = Q^{- 1} R^{\top} (u - u^{\star}),
\end{equation}
 and consider the preconditioned update $u^+ = u - P^{- 1} \nabla_u \lambda (u, v)$. It gives
\begin{align}
  u^+ - u^{\star} ={} & u - P^{- 1} \nabla_u \lambda (u, v) - u^{\star}
  \nonumber\\
  ={} & u - u^{\star} - P^{- 1} (P (u - u^{\star}) - R (v - v^{\star})) \nonumber \\
  ={} & P^{- 1} R (v - v^{\star}), \label{eqn:contraction-2}
\end{align}
and it is easy to see that $\nabla_u \lambda (u^+, v) = P (u^+ - u^{\star})
- R (v - v^{\star}) = 0$. Hence {\am} applied to $\lambda$ yields the same $u^+$.

Next, we consider the contraction, and we successively deduce that
\begin{align}
  \lambda (u, v) - \varphi (u^{\star}, v^{\star}) ={} & \tfrac{1}{2} \| u -
  u^{\star} \|_P^2 - \langle u - u^{\star}, R (v - v^{\star}) \rangle +
  \tfrac{1}{2} \| v - v^{\star} \|^2_Q \nonumber\\
  ={} & \tfrac{1}{2} \| u - u^{\star} \|_{P - R Q^{- 1} R^{\top}}^2, \label{eqn:contraction-3}
\end{align}
where \eqref{eqn:contraction-3} applies \eqref{eqn:contraction-1}. Furthermore, plugging \eqref{eqn:contraction-1} into  \eqref{eqn:contraction-2} gives $u^+ - u^{\star} = P^{- 1} R Q^{- 1} R^{\top} (u -
u^{\star})$ and 
\begin{align}
\lambda (u^+, v) - \varphi (u^{\star}, v^{\star}) ={} & \tfrac{1}{2} \| u^+ - u^{\star} \|_P^2 - \langle u^+ - u^{\star}, R (v -
  v^{\star}) \rangle + \tfrac{1}{2} \| v - v^{\star} \|^2_Q \nonumber\\
  ={} & \tfrac{1}{2} \| u - u^{\star} \|_{R Q^{- 1} R^{\top} - R Q^{- 1}
  R^{\top} P^{- 1} R Q^{- 1} R^{\top}}^2 . \nonumber
\end{align}
Now it suffices to show, for any $z$, that
\[\| z \|_{R Q^{- 1} R^{\top} - R Q^{- 1} R^{\top} P^{- 1} R
Q^{- 1} R^{\top}}^2 \leq (1 - \sigma_2) \| z \|_{P - R Q^{- 1} R^{\top}}^2.\]

Denote the eigen-decomposition $M \assign P^{- 1 / 2} R Q^{- 1} R^{\top} P^{- 1 / 2} = U \Sigma
U^{\top}$ and suppose the eigenvalues of $\Sigma$ are non-increasing (the top-left entry is the largest). We may check that $P^{1 / 2} \1_m$ is an eigenvector with
eigenvalue $1$:
\[ M P^{1 / 2} \1_m = P^{- 1 / 2} R Q^{- 1} R^{\top} P^{- 1 / 2} ( P^{1 /
   2} \1_m ) = P^{1 / 2} \1_m, \]
and therefore $P^{1 / 2} \1_m \in \tmop{Nul} (I - M)$. Besides, note that $0
\preceq \Sigma \preceq I$ since $\big(\begin{smallmatrix}
  P & - R\\
  - R^{\top} & Q
\end{smallmatrix}\big) \succeq 0$ implies that its Schur complement $P - R Q^{- 1}
R^{\top} \succeq 0$ is positive semidefinite, which further implies
\[ I \succeq P^{- 1 / 2} R Q^{- 1} R^{\top} P^{- 1 / 2} \succeq 0 . \]
Now we deduce that
\begin{align}
  \| z \|_{R Q^{- 1} R^{\top} - R Q^{- 1} R^{\top} P^{- 1} R Q^{- 1}
  R^{\top}}^2 ={} & \langle z, (R Q^{- 1} R^{\top} - R Q^{- 1} R^{\top} P^{- 1}
  R Q^{- 1} R^{\top}) z \rangle \nonumber\\
  ={} & \langle z, (P^{1 / 2} M P^{1 / 2} - P^{1 / 2} M^{\top} M P^{1 / 2}) z
  \rangle \nonumber\\
  ={} & \| P^{1 / 2} z \|_{M - M^{\top} M}^2, \nonumber
\end{align}
and $\| z \|_{P - R Q^{- 1} R^{\top}}^2 = \langle z, (P - P^{1 / 2} M P^{1 /
2}) z \rangle = \| P^{1 / 2} z \|_{I - M}^2$. Since $\tmop{Nul} (I - M)
\subseteq \tmop{Nul} (M - M^{\top} M)$, we assume $P^{1 / 2} z \perp
\tmop{Nul} (I - M)$ and deduce that
\begin{align}
  \max_{P^{1 / 2} z \perp \tmop{Nul} (I - M)}  \tfrac{\| P^{1 / 2} z \|_{M -
  M^{\top} M}^2}{\| P^{1 / 2} z \|_{I - M}^2} ={} & \max_{x \perp \tmop{Nul} (I
  - M)}  \tfrac{\langle x, (M - M^{\top} M) x \rangle }{\langle x, (I - M) x
  \rangle} \nonumber\\
  ={} & \max_{x \perp \tmop{Nul} (I - \Sigma)}  \tfrac{\langle x, (\Sigma -
  \Sigma^2 ) x \rangle }{\langle x, (I - \Sigma) x \rangle} \nonumber\\
  ={} & \max_{x_1 = 0, x \perp \tmop{Nul} (I - \Sigma)}  \tfrac{\langle x,
  (\Sigma - \Sigma^2 ) x \rangle }{\langle x, (I - \Sigma) x \rangle}
  \nonumber\\
  \leq{} & \lambda_{m-1} (\Sigma) \nonumber\\
  ={} & 1 - \lambda_2 (I - \Sigma) = 1 - \sigma_2 . \nonumber
\end{align}

Using $\| z \|_{R Q^{- 1} R^{\top} - R Q^{- 1} R^{\top} P^{- 1} R Q^{- 1}
R^{\top}}^2 \leq (1 - \sigma_2) \| z \|_{P - R Q^{- 1} R^{\top}}^2$, we have
\[ \lambda (u^+, v) - \varphi (u^{\star}, v^{\star}) \leq (1 - \sigma_2)
   [\lambda (u, v) - \varphi (u^{\star}, v^{\star})], \]
which proves the first contraction bound. Repeating the argument by switching the roles of
$u, v$ and noticing that
\begin{align}
  \lambda_2 (I_m - P^{- 1 / 2} R Q^{- 1} R^{\top} P^{- 1 / 2}) ={} & \lambda_2
  (I_m - P^{- 1 / 2} R Q^{- 1 / 2} (P^{- 1 / 2} R Q^{- 1 / 2})^{\top})
  \nonumber\\
  ={} & \lambda_2 (I_n - (P^{- 1 / 2} R Q^{- 1 / 2})^{\top} P^{- 1 / 2} R Q^{- 1
  / 2}) \nonumber\\
  ={} & \lambda_2 (I_n - Q^{- 1 / 2} R^{\top} P^{- 1} R Q^{- 1 / 2}), \nonumber
\end{align}

we have $\lambda (u^+, v^+) - \varphi (u^{\star}, v^{\star}) \leq (1 -
\sigma_2)^2 [\lambda (u, v) - \varphi (u^{\star}, v^{\star})]$. This completes
the proof.

\subsection{Proof of Lemma \ref{lem:grad-hess-new}}

Recall the notation $\Delta_{i j} \assign \Delta u_i - \Delta v_j$ and define
$g (t) \assign \varphi (u + t d^u, v + t d^v)$ and $h (t) \assign \lambda (u +
t d^u, v + t d^v)$ with $d = (d^u, d^v) \in \mathbb{R}^{m + n}$. We have $g
(0) = \varphi (u, v)$, $h (0) = \lambda (u, v)$, and that
\begin{align}
  g' (t) ={} & \tfrac{\mathd}{\mathd t} [ \textstyle\sum_{i, j} a_{i j}^{\star} \phi
  (\Delta_{i j} + t (d^u_i - d^v_j))] \nonumber\\
  ={} & \textstyle \sum_{i, j} a_{i j}^{\star} \phi' (\Delta_{i j} + t (d^u_i - d^v_j))
  (d^u_i - d^v_j), \nonumber\\
  g'' (t) ={} & \textstyle\sum_{i, j} a_{i j}^{\star} \phi'' (\Delta_{i j} + t (d^u_i -
  d^v_j)) (d^u_i - d^v_j)^2, \nonumber\\
  h' (t) ={} & \textstyle\sum_{i, j} a_{i j}^{\star} (\Delta_{i j} + t (d^u_i - d^v_j))
  (d^u_i - d^v_j), \nonumber\\
  h'' (t) ={} & \textstyle\sum_{i, j} a_{i j}^{\star} (d^u_i - d^v_j)^2 . \nonumber
\end{align}
Taking $t = 0$, we have
\begin{align}
  g' (0) ={} & \textstyle \langle \nabla \varphi (u, v), d \rangle = \sum_{i, j} a_{i
  j}^{\star} \phi' (\Delta_{i j}) (d^u_i - d^v_j), \nonumber\\
  g'' (0) ={} & \textstyle \langle d, \nabla^2 \varphi (u, v) d \rangle ={} \sum_{i, j} a_{i
  j}^{\star} \phi'' (\Delta_{i j}) (d^u_i - d^v_j)^2, \nonumber\\
  h' (0) ={} & \textstyle \langle \nabla \lambda (u, v), d \rangle = \sum_{i, j} a_{i
  j}^{\star} \Delta_{i j} (d^u_i - d^v_j), \nonumber\\
  h'' (0) ={} & \textstyle\langle d, \nabla^2 \lambda (u^{\star}, v^{\star}) d \rangle =
  \sum_{i, j} a_{i j}^{\star} (d^u_i - d^v_j)^2 . \nonumber
\end{align}

\paragraph{Proof of relation 1.}We invoke \Cref{lem:phi} and deduce that
\begin{align}
  | \langle \nabla \varphi (u, v) - \nabla \lambda (u, v), d \rangle | ={} &
  | \textstyle \sum_{i, j} a_{i j}^{\star} \phi' (\Delta_{i j}) (d^u_i - d^v_j) -
  \textstyle \sum_{i, j} a_{i j}^{\star} \Delta_{i j} (d^u_i - d^v_j)| \nonumber\\
  ={} & | \textstyle \sum_{i, j} a_{i j}^{\star} [\phi' (\Delta_{i j}) - \Delta_{i j}]
  (d^u_i - d^v_j)| \nonumber \\
  ={} & | \textstyle \sum_{i, j} a_{i j}^{\star} \phi (\Delta_{i j}) (d^u_i - d^v_j)
 | \label{eqn:grad-hess-new-1} \\
  \leq{} & 2 [ \textstyle \sum_{i, j} a_{i j}^{\star} \phi (\Delta_{i j})] \| d
  \|_{\infty} \nonumber\\
  ={} & 2 \varepsilon \| d \|_{\infty}, \label{eqn:grad-hess-new-2}
\end{align}
where \eqref{eqn:grad-hess-new-1} uses $\phi(\delta) = \phi'(\delta) - \delta$ from \Cref{lem:phi}; \eqref{eqn:grad-hess-new-2} uses \Cref{prop-equivalence}. Taking $d = \tfrac{\nabla \varphi (u, v) - \nabla \lambda (u, v)}{\| \nabla
\varphi (u, v) - \nabla \lambda (u, v) \|}$ gives $\| \nabla \varphi (u, v) -
\nabla \lambda (u, v) \| \leq 2 \varepsilon \tfrac{\| d \|_{\infty}}{\| d \|}
\leq 2 \varepsilon$ and that
\[ \| \nabla \varphi (u, v) - \nabla \lambda (u, v) \|_{S^{- 1}} ={} \| S^{- 1
   / 2} [\nabla \varphi (u, v) - \nabla \lambda (u, v)] \| \leq \tfrac{2
   \varepsilon}{\sqrt{\| (p, q) \|_{- \infty}}} . \]
Next, consider
\begin{align}
  | \langle \nabla \varphi (u, v), S^{- 1 / 2} d \rangle | ={} & | \textstyle \sum_{i,
  j} a_{i j}^{\star} \phi' (\Delta_{i j}) ( \tfrac{d^u_i}{\sqrt{p_i}} -
  \tfrac{d^v_j}{\sqrt{q_j}})| \nonumber\\
  \leq{} & \textstyle \sum_{i, j} a_{i j}^{\star} | \phi' (\Delta_{i j}) | |
  \tfrac{d^u_i}{\sqrt{p_i}} - \tfrac{d^v_j}{\sqrt{q_j}}| \nonumber\\
  \leq{} & \textstyle \sum_{i, j} a_{i j}^{\star} ( \sqrt{2 \phi (\Delta_{i j})} +
  \phi (\Delta_{i j})) | \tfrac{d^u_i}{\sqrt{p_i}} -
  \tfrac{d^v_j}{\sqrt{q_j}}| \label{eqn:grad-hess-new-3} \\
  ={} & \textstyle \sum_{i, j} a_{i j}^{\star} \sqrt{2 \phi (\Delta_{i j})} |
  \tfrac{d^u_i}{\sqrt{p_i}} - \tfrac{d^v_j}{\sqrt{q_j}}| + \textstyle \sum_{i, j}
  a_{i j}^{\star} \phi (\Delta_{i j}) | \tfrac{d^u_i}{\sqrt{p_i}} -
  \tfrac{d^v_j}{\sqrt{q_j}}|, \nonumber
\end{align}
where \eqref{eqn:grad-hess-new-3} invokes $|\phi'(\delta)| \leq \sqrt{2\phi(\delta)} + \phi(\delta)$ from \Cref{lem:phi}. Now we bound these two terms respectively by
\begin{align}
  \textstyle \sum_{i, j} a_{i j}^{\star} \sqrt{2 \phi (\Delta_{i j})} |
  \tfrac{d^u_i}{\sqrt{p_i}} - \tfrac{d^v_j}{\sqrt{q_j}}| ={} & \textstyle \sum_{i,
  j} \sqrt{2 a_{i j}^{\star} \phi (\Delta_{i j})} \Big| \tfrac{\sqrt{a_{i
  j}^{\star}} d^u_i}{\sqrt{\sum_k a_{i k}^{\star}}} - \tfrac{\sqrt{a_{i
  j}^{\star}} d^v_j}{\sqrt{\sum_k a_{k j}^{\star}}} \Big| \nonumber\\
  \leq{} & \sqrt{\textstyle \sum_{i, j} 2 a_{i j}^{\star} \phi (\Delta_{i j})}
  \sqrt{\textstyle \sum_{i, j} \Big( \tfrac{\sqrt{a_{i j}^{\star}} d^u_i}{\sqrt{\sum_k
  a_{i k}^{\star}}} - \tfrac{\sqrt{a_{i j}^{\star}} d^v_j}{\sqrt{\sum_k a_{k
  j}^{\star}}}\Big)^2} \label{eqn:grad-hess-new-4}\\
  ={} & \sqrt{2 \varepsilon} \sqrt{\textstyle \sum_{i, j} \Big( \tfrac{\sqrt{a_{i
  j}^{\star}} d^u_i}{\sqrt{\sum_k a_{i k}^{\star}}} - \tfrac{\sqrt{a_{i
  j}^{\star}} d^v_j}{\sqrt{\sum_k a_{k j}^{\star}}}\Big)^2} \nonumber\\
  \leq{} & \sqrt{2 \varepsilon} \sqrt{2 \Big[ \textstyle \sum_i \textstyle \sum_j \tfrac{a_{i
  j}^{\star} (d^u_i)^2}{ \sum_k a_{i k}^{\star}} + \textstyle \sum_j \textstyle \sum_i \tfrac{a_{i
  j}^{\star} (d^v_j)^2}{ \sum_k a_{k j}^{\star}}\Big]}  \label{eqn:grad-hess-new-5}\\
  ={} & 2 \sqrt{\varepsilon} \sqrt{\textstyle \sum_i (d^u_i)^2 + \textstyle \sum_j (d^v_j)^2} = 2
  \sqrt{\varepsilon} \| d \|, \nonumber
\end{align}
where \eqref{eqn:grad-hess-new-4} uses Cauchy-Schwarz $\sum_{i, j} a_{i j} b_{i j} \leq \sqrt{\sum_{i, j} a_{i j}^2} \sqrt{\sum_{i, j} b_{i j}^2}$; \eqref{eqn:grad-hess-new-5} uses $(a - b)^2 \leq 2 a^2 + 2 b^2$.
Moreover, we have $\textstyle \sum_{i, j} a_{i j}^{\star} \phi (\Delta_{i j}) |
\tfrac{d^u_i}{\sqrt{p_i}} - \tfrac{d^v_j}{\sqrt{q_j}}| \leq \varepsilon
\max_{i j} | \tfrac{d^u_i}{\sqrt{p_i}} - \tfrac{d^v_j}{\sqrt{q_j}}
| \leq{} \tfrac{2 \varepsilon \| d \|_{\infty}}{\sqrt{\| (p, q) \|_{-
\infty}}}$. Taking $d = \tfrac{S^{- 1 / 2} \nabla \varphi (u, v)}{\| S^{- 1 /
2} \nabla \varphi (u, v) \|}$ gives
\[ \| \nabla \varphi (u, v) \|_{S^{- 1}} = | \langle \nabla \varphi (u, v),
   S^{- 1 / 2} d \rangle | \leq 2 \sqrt{\varepsilon} + \tfrac{2
   \varepsilon}{\sqrt{\| (p, q) \|_{- \infty}}}. \]
Before concluding, we also prove a bound on $\| \nabla \varphi (u, v) \|$.
Consider

\begin{align}
  | \langle \nabla \varphi (u, v), d \rangle | ={} & \textstyle | \sum_{i, j} a_{i
  j}^{\star} \phi' (\Delta_{i j}) (d^u_i - d^v_j) | \nonumber\\
  \leq{} & \textstyle \sum_{i, j} a_{i j}^{\star} | \phi' (\Delta_{i j}) | | d^u_i - d^v_j
  | \nonumber\\
  \leq{} & \textstyle 2 \sum_{i, j} a_{i j}^{\star} [ \sqrt{2 \phi (\Delta_{i j})} +
  \phi (\Delta_{i j}) ] \| d \|_{\infty} \label{eqn:grad-hess-new-5-1} \\
  ={} & \textstyle 2 [ \sum_{i, j} a_{i j}^{\star} \sqrt{2 \phi (\Delta_{i j})} +
  \sum_{i, j} a_{i j}^{\star} \phi (\Delta_{i j}) ] \| d \|_{\infty}
  \nonumber\\
  ={} & \textstyle 2 [ \sqrt{2} \sum_{i, j} \sqrt{a_{i j}^{\star}} \sqrt{a_{i
  j}^{\star} \phi (\Delta_{i j})} + \sum_{i, j} a_{i j}^{\star} \phi
  (\Delta_{i j}) ] \| d \|_{\infty} \nonumber\\
  \leq{} & \textstyle 2 [ \sqrt{2} \sqrt{\sum_{i, j} a_{i j}^{\star}} \sqrt{\sum_{i,
  j} a_{i j}^{\star} \phi (\Delta_{i j})} + \sum_{i, j} a_{i j}^{\star} \phi
  (\Delta_{i j}) ] \| d \|_{\infty} \label{eqn:grad-hess-new-5-2}\\
  \leq{} & \textstyle ( 4 \sqrt{\| p \|_1 \varepsilon} + 2 \varepsilon ) \| d
  \|_{\infty}, \label{eqn:grad-hess-new-5-3}
\end{align}
where \eqref{eqn:grad-hess-new-5-1} uses \Cref{lem:phi}; \eqref{eqn:grad-hess-new-5-2} uses Cauchy-Schwarz; \eqref{eqn:grad-hess-new-5-3} uses $\|A^\star\|_1 = \|p\|_1$.  Taking $d = \tfrac{\nabla \varphi (u, v)}{\| \nabla \varphi (u, v) \|}$ gives
$\| \nabla \varphi (u, v) \| \leq 4 \sqrt{\| p \|_1 \varepsilon} + 2
\varepsilon$.
\paragraph{Proof of relation 2.}
We deduce that
\begin{align}
  | \langle d, S^{- 1 / 2} [\nabla^2 \varphi (u, v) - \nabla^2 \lambda (u,
  v)] S^{- 1 / 2} d \rangle |
  ={} & | \langle S^{- 1 / 2} d, [\nabla^2 \varphi (u, v) - \nabla^2 \lambda (u,
  v)] S^{- 1 / 2} d \rangle | \nonumber\\
  ={} & \big| \textstyle \sum_{i, j} a_{i j}^{\star} [\phi'' (\Delta_{i j}) - 1] \big(
  \tfrac{d^u_i}{\sqrt{p_i}} - \tfrac{d^v_j}{\sqrt{q_j}}\big)^2\big|
  \nonumber\\
  \leq{} & \textstyle \sum_{i, j} a_{i j}^{\star} \big| \phi'' (\Delta_{i j}) - 1 \big| \big(
  \tfrac{d^u_i}{\sqrt{p_i}} - \tfrac{d^v_j}{\sqrt{q_j}}\big)^2 \nonumber\\
  \leq{} & \textstyle \sum_{i, j} a_{i j}^{\star} \sqrt{2 \phi (\Delta_{i j})} \big|
  \tfrac{d^u_i}{\sqrt{p_i}} - \tfrac{d^v_j}{\sqrt{q_j}}\big|^2 + \textstyle \sum_{i,
  j} a_{i j}^{\star} \phi (\Delta_{i j}) \big| \tfrac{d^u_i}{\sqrt{p_i}} -
  \tfrac{d^v_j}{\sqrt{q_j}}\big|^2, \nonumber
\end{align}
where the last inequality uses $\phi''(\delta) - 1 = \phi'(\delta)$ and \Cref{lem:phi}. Hence, we can similarly bound
\begin{align}
& \textstyle\sum_{i, j} a_{i j}^{\star} \sqrt{2 \phi (\Delta_{i j})} |
  \tfrac{d_i^u}{\sqrt{p_i}} - \tfrac{d_j^v}{\sqrt{q_j}} |^2 \\
  ={} & \textstyle \sum_{i, j} \sqrt{2 a^{\star}_{i j} \phi (\Delta_{i j})}
  \sqrt{a^{\star}_{i j}} | \tfrac{d_i^u}{\sqrt{p_i}} -
  \tfrac{d_j^v}{\sqrt{q_j}} |^2 \nonumber\\
  \leq{} & 2\textstyle \sum_{i, j} \sqrt{2 a^{\star}_{i j} \phi (\Delta_{i j})} \Big[
  \frac{\sqrt{a^{\star}_{i j}} (d_i^u)^2}{\sqrt{\sum_k a_{i k}^{\star}}}
  \frac{1}{\sqrt{p_i}} + \frac{\sqrt{a^{\star}_{i j}} (d_j^v)^2}{\sqrt{\sum_k
  a_{k j}^{\star}}} \frac{1}{\sqrt{q_j}} \Big] \label{eqn:grad-hess-new-6-0} \\
  \leq{} & \textstyle \frac{2}{\sqrt{\| (p, q) \|_{- \infty}}}  \sum_{i, j} \sqrt{2 a^{\star}_{i j} \phi
  (\Delta_{i j})} \Big[ \frac{\sqrt{a^{\star}_{i j}} (d_i^u)^2}{\sqrt{\sum_k
  a_{i k}^{\star}}} + \frac{\sqrt{a^{\star}_{i j}} (d_j^v)^2}{\sqrt{\sum_k
  a_{k j}^{\star}}} \Big] \label{eqn:grad-hess-new-6-1}\\
  \leq{} & \textstyle \frac{2\sqrt{2}}{\sqrt{\| (p, q) \|_{- \infty}}} \sqrt{\sum_{i, j} a^{\star}_{i j} \phi
  (\Delta_{i j})} \sqrt{\sum_i \sum_j \frac{a^{\star}_{i j} (d_i^u)^4}{\sum_k
  a_{i k}^{\star}} + \sum_j \sum_i \frac{a^{\star}_{i j} (d_j^v)^4}{\sum_k
  a_{k j}^{\star}}} \label{eqn:grad-hess-new-6-2}\\
  ={} & \textstyle \frac{2\sqrt{2}}{\sqrt{\| (p, q) \|_{- \infty}}} \sqrt{\sum_{i, j} a^{\star}_{i j} \phi
  (\Delta_{i j})} \sqrt{\sum_i (d_i^u)^4 + \sum_j (d_j^v)^4} = \frac{2\sqrt{2
  \varepsilon}}{\sqrt{\| (p, q) \|_{- \infty}}} \| d \|^2_4 \leq \frac{2\sqrt{2
  \varepsilon}}{\sqrt{\| (p, q) \|_{- \infty}}} \| d \|^2,  \label{eqn:grad-hess-new-7}
\end{align}
where \eqref{eqn:grad-hess-new-6-0} uses $(a - b)^2 \leq 2 a^2 + 2 b^2$; \eqref{eqn:grad-hess-new-6-1} uses $p_i, q_j \geq \| (p, q) \|_{- \infty}$; \eqref{eqn:grad-hess-new-6-2} applies Cauchy-Schwarz and \eqref{eqn:grad-hess-new-7} uses $\|x\|^2_4 \leq \|x\|^2$. We also have $\textstyle \sum_{i, j} a_{i j}^{\star} \phi (\Delta_{i j}) |
\tfrac{d^u_i}{\sqrt{p_i}} - \tfrac{d^v_j}{\sqrt{q_j}}|^2 \leq \tfrac{2
\varepsilon}{{\| (p, q) \|_{- \infty}}} \| d \|_{\infty}^2$. Hence for
any $d \neq 0$, we have
\[ | \langle d, S^{- 1 / 2} [\nabla^2 \varphi (u, v) - \nabla^2 \lambda (u,
   v)] S^{- 1 / 2} d \rangle | \leq 2 \sqrt{\tfrac{2\varepsilon}{\| (p, q) \|_{- \infty}}} \| d \|^2 + \tfrac{2
   \varepsilon}{{\| (p, q) \|_{- \infty}}} \| d \|_{\infty}^2 \]
and $\| S^{- 1 / 2} [\nabla^2 \varphi (u, v) - \nabla^2 \lambda (u, v)] S^{- 1
/ 2} \| \leq 2 \sqrt{\tfrac{2\varepsilon}{\| (p, q) \|_{- \infty}}} + \tfrac{2 \varepsilon}{{\| (p, q)
\|_{- \infty}}}$.

\paragraph{Proof of relation 3.}

Finally, we consider the zeroth-order relation. The proof uses the fact that $\nabla \varphi \approx 0$ to restrict $(u, v)$ and deduce a bound that only depends on $\sigma_2$. Since $\varepsilon \leq
\tfrac{\sigma_2^2 s}{1296} \leq s$, we have $\tfrac{1}{\sqrt{s}} \varepsilon \leq
\sqrt{\varepsilon}$ and that $\sqrt{\tfrac{\varepsilon}{s}} \geq \tfrac{\varepsilon}{s}$. By Taylor's
theorem, for any optimal $z^{\star} = ( \us, \vs )$, we define $z_t
\assign ( \us + t ( u - \us ), \vs + t ( v - \vs )
)$ and deduce
\begin{align}
  \varphi (u, v) - \lambda (u, v) ={} & \textstyle \int_0^1 (1 - t) \langle z - z^{\star},
  [\nabla^2 \varphi (z_t) - \nabla^2 \varphi (z^{\star})] (z - z^{\star})
  \rangle \mathd t \nonumber\\
  ={} & \textstyle \int_0^1 (1 - t) \langle S^{1 / 2} (z - z^{\star}), S^{- 1 / 2}
  [\nabla^2 \varphi (z_t) - \nabla^2 \varphi (z^{\star})] S^{- 1 / 2} S^{1 /
  2} (z - z^{\star}) \rangle \mathd t \nonumber\\
  \geq{} & - ( 2\sqrt{\tfrac{\varepsilon}{s}} + \tfrac{\varepsilon}{s} ) \|z- z^\star\| _S^2 \geq -3\sqrt{\tfrac{\varepsilon}{s}} \|z- z^\star\| _S^2. \label{eqn:grad-hess-new-8-5}
\end{align}
Next, using $\| \nabla \varphi (z) - \nabla \lambda (z) \|_{S^{- 1}} \leq
\tfrac{2}{\sqrt{s}} \varepsilon$ and relation (1) from \Cref{lem:grad-hess-new}, we have
\[ \| \nabla \lambda (z) \|_{S^{- 1}} \leq \| \nabla \varphi (z) \|_{S^{- 1}}
   + \| \nabla \lambda (z) - \nabla \varphi (z) \|_{S^{- 1}} \leq 2
   \sqrt{\varepsilon} + \tfrac{2}{\sqrt{s}} \varepsilon + \tfrac{2}{\sqrt{s}}
   \varepsilon \leq 6 \sqrt{\varepsilon} . \]
Given that
\[ \nabla \lambda (z) = \Big(\begin{smallmatrix}
     P ( u - \us ) - R ( v - \vs )\\
     Q ( v - \vs ) - R^{\top} ( u - \us )
   \end{smallmatrix}\Big), \]
without loss of generality, we let $(\delta_u, \delta_v) = \nabla \lambda (z)$
and
\begin{align}
  P (u - u^{\star}) ={} & R (v - v^{\star}) + \delta_u \nonumber\\
  Q ( v - \vs ) ={} & R^{\top} ( u - \us ) + \delta_v,
  \nonumber
\end{align}
and that $\| \delta_u \|_{P^{- 1}} \leq \| (\delta_u, \delta_v) \|_{S^{- 1}} = \|  \nabla \lambda (z) \|_{S^{- 1}} \leq 6 \sqrt{\varepsilon}$.
Next, we lower bound $\lambda (z) - \lambda (z^{\star})$ and deduce that
\begin{align}
  \lambda (u, v) - \lambda ( \us, \vs ) ={} & \tfrac{1}{2} \| u
  - \us \|_P^2 - \langle u - \us, R ( v - \vs )
  \rangle + \tfrac{1}{2} \| v - \vs \|_Q^2 \nonumber\\
  ={} & \tfrac{1}{2} \| \delta_u \|_{P^{- 1}}^2 + \tfrac{1}{2} \| v - \vs
  \|_Q^2 - \tfrac{1}{2} \| v - \vs \|_{R^{\top} P^{- 1} R}^2
  \label{eqn:grad-hess-new-8} \\
  ={} & \tfrac{1}{2} \| \delta_u \|_{P^{- 1}}^2 + \tfrac{1}{2} \| v - \vs
  \|_{Q - R^{\top} P^{- 1} R}^2,   \label{eqn:grad-hess-new-9}
\end{align}
where \eqref{eqn:grad-hess-new-8} plugs in $u - u^{\star} = P^{-1} (R (v - v^{\star}) + \delta_u)$. 
Since for any optimal solution $( \us, \vs )$, $( \us + \alpha \1_m, \vs + \alpha \1_n )$ is also optimal,
we can pick $( \us, \vs )$ such that $v - \vs \perp q \Leftrightarrow Q^{1/2} (v-\vs) \perp Q^{1/2}\1_n$. With this choice, we have
\begin{align}
  \tfrac{1}{2} \| v - \vs \|_{Q - R^{\top} P^{- 1} R}^2 ={} &
  \tfrac{1}{2} \| Q^{1 / 2} ( v - \vs ) \|_{I - Q^{- 1 /
  2} R^{\top} P^{- 1} R Q^{- 1 / 2}}^2 \nonumber\\
  \geq{} & \min_{ w \perp Q^{1/2}\1_n, w \neq 0} \tfrac{1}{2} \tfrac{\langle w, (I -
  Q^{- 1 / 2} R^{\top} P^{- 1} R Q^{- 1 / 2}) w \rangle}{\| w \|^2}  \| v
  - \vs \|^2_Q \nonumber\\
  ={} & \tfrac{\sigma_2}{2}  \| v - \vs \|^2_Q . \label{eqn:grad-hess-new-9-5}
\end{align}
Next, we deduce that
\begin{align}
  \varepsilon \geq{} & \varphi (u, v) - \varphi ( \us, \vs )
  \nonumber\\
  ={} & \varphi (u, v) - \lambda (u, v) + \lambda (u, v) - \varphi ( \us,
  \vs ) \nonumber\\
  \geq{} & - 3\sqrt{\tfrac{\varepsilon}{s}}\| z - z^{\star} \|_S^2 + \tfrac{1}{2} \|
  \delta_u \|_{P^{- 1}}^2 + \tfrac{1}{2} \| v - \vs \|_{Q -
  R^{\top} P^{- 1} R}^2 \label{eqn:grad-hess-new-10} \\
  \geq{} & - 3\sqrt{\tfrac{\varepsilon}{s}} \| z - z^{\star} \|_S^2 + \tfrac{1}{2} \|
  \delta_u \|_{P^{- 1}}^2 + \tfrac{\sigma_2}{2}  \| v - \vs \|^2_Q \label{eqn:grad-hess-new-11},
\end{align}
where \eqref{eqn:grad-hess-new-10} plugs in \eqref{eqn:grad-hess-new-8-5} and \eqref{eqn:grad-hess-new-9}; \eqref{eqn:grad-hess-new-11} applies \eqref{eqn:grad-hess-new-9-5}. Using $P (u - u^{\star}) = R (v - v^{\star}) + \delta_u$, we have
\begin{align}
  \| u - u^{\star} \|^2_P ={} & \| P^{- 1} R (v - v^{\star}) + P^{- 1} \delta_u
  \|^2_P \nonumber\\
  ={} & \langle P^{- 1} R (v - v^{\star}) + P^{- 1} \delta_u, R (v - v^{\star})
  + \delta_u \rangle \nonumber\\
  ={} & \| v - v^{\star} \|_{R^{\top} P^{- 1} R}^2 + 2 \langle P^{- 1 / 2} R (v
  - v^{\star}), P^{- 1 / 2} \delta_u \rangle + \| \delta_u \|_{P^{- 1}}^2
  \nonumber\\
  \leq{} & (1 + \theta) \| v - v^{\star} \|_{R^{\top} P^{- 1} R}^2 + ( 1 +
  \tfrac{1}{\theta} ) \| \delta_u \|_{P^{- 1}}^2 \nonumber
\end{align}
for any $\theta > 0$. Hence
\begin{align}
  \| z - z^{\star} \|_S^2 ={} & \| u - u^{\star} \|^2_P + \| v - \vs
  \|^2_Q \nonumber\\
  \leq{} & (1 + \theta) \| v - v^{\star} \|_{R^{\top} P^{- 1} R}^2 + \| v - \vs
  \|^2_Q + ( 1 +
  \tfrac{1}{\theta} ) \| \delta_u \|_{P^{- 1}}^2 \nonumber\\
  ={} & (1 + \theta) \| Q^{1 / 2} (v - v^{\star}) \|_{Q^{-1 / 2} R^{\top} P^{- 1}
  R Q^{-1 / 2}}^2 + \| v - \vs \|^2_Q + ( 1 + \tfrac{1}{\theta} ) \| \delta_u \|_{P^{-
  1}}^2 \nonumber\\
  \leq{} & (2 + \theta) \| v - v^{\star} \|^2_Q + ( 1 + \tfrac{1}{\theta}
  ) \| \delta_u \|_{P^{- 1}}^2, \nonumber
\end{align}
since $I \succeq Q^{- 1 / 2} R^{\top} P^{- 1} R Q^{- 1 / 2}$. Plugging the relation back into \eqref{eqn:grad-hess-new-11}, we have
\begin{align}
  \varepsilon \geq{} & - 3\sqrt{\tfrac{\varepsilon}{s}}\| z - z^{\star} \|_S^2 + \tfrac{1}{2} \| \delta_u \|_{P^{- 1}}^2 +
  \tfrac{\sigma_2}{2}  \| v - \vs \|^2_Q \nonumber\\
  \geq{} & - 3\sqrt{\tfrac{\varepsilon}{s}} \big[
  (2 + \theta) \| v - v^{\star} \|^2_Q + ( 1 + \tfrac{1}{\theta} )
  \| \delta_u \|_{P^{- 1}}^2 \big] + \tfrac{1}{2} \| \delta_u \|_{P^{- 1}}^2
  + \tfrac{\sigma_2}{2}  \| v - \vs \|^2_Q \nonumber\\
  ={} & \big[ \tfrac{\sigma_2}{2} - (6 + 3\theta) \sqrt{\tfrac{\varepsilon}{s}} \big] \| v - \vs \|^2_Q +
  \big[ \tfrac{1}{2} - 3\sqrt{\tfrac{\varepsilon}{s}} ( 1 + \tfrac{1}{\theta}
  ) \big] \| \delta_u \|_{P^{- 1}}^2 . \nonumber
\end{align}
Taking $\theta = 1$ and using $\varepsilon \leq \tfrac{s \sigma_2^2}{1296}$, we have
\[ \tfrac{1}{2} - 6\sqrt{\tfrac{\varepsilon}{s}}\geq 0, \quad \text{and} \quad 9 \sqrt{\tfrac{\varepsilon}{s}} \leq
   \tfrac{\sigma_2}{4}, \]
   giving $\varepsilon \geq \tfrac{\sigma_2}{4} \| v - \vs \|^2_Q$
and that $\| v - \vs \|^2_Q \leq \tfrac{4 \varepsilon}{\sigma_2}$. Finally, 
\[ \| z - z^{\star} \|_S^2 \leq 3 \| v - v^{\star} \|^2_Q + 2 \| \delta_u
   \|_{P^{- 1}}^2 \leq \tfrac{12}{\sigma_2} \varepsilon + 72 \varepsilon = 12
   ( \tfrac{1}{\sigma_2} + 6 ) \varepsilon \leq \tfrac{84}{\sigma_2}\varepsilon, \]
and that
\[ | \varphi (u, v) - \lambda (u, v) | \leq ( 2\sqrt{\tfrac{\varepsilon}{s}} + \tfrac{\varepsilon}{s} ) \| z - z^{\star} \|_S^2 \leq \tfrac{168}{\sigma_2} (\tfrac{1}{\sqrt{s}} \varepsilon^{3 / 2} +
   \tfrac{1}{s} \varepsilon^2 ) . \]
This completes the proof.

\begin{rem} \label{rem:symm}
	Symmetrically, we can also show that $\|  u - \us \|^2_P \leq \tfrac{4 \varepsilon}{\sigma_2}$ by picking $\us$ such that $u - \us \perp p$. But note that we may not simultaneously enforce both $u - \us \perp p$ and $v - \vs \perp q$ with the same $z^\star$.
\end{rem}

\subsection{Proof of Lemma \ref{lem:global-conv}}
Let $\varepsilon = \varphi(u,v) - \varphi(\us, \vs)$. The result is immediate if $\varepsilon = 0$, so suppose $\varepsilon > 0$. Since $(u, v)$ follows a minimization over $v$, we have $\nabla_v \varphi (u, v) = 0$.
Fix $v$ and define $\alpha (\cdot) \assign \varphi (\cdot, v)$ and $\ell
(\cdot) \assign \lambda (\cdot, v)$. We have $\nabla^2 \alpha (u) =
\nabla^2_{u u} \varphi (u, v)$.
By \Cref{lem:grad-hess-new}, we have
\begin{align}
  \| \nabla \alpha (u) - \nabla \ell (u) \|_{P^{- 1}} ={} & \| \nabla_u \varphi
  (u, v) - \nabla_u \lambda (u, v) \|_{P^{- 1}} \nonumber\\
  \leq{} & \| \nabla \varphi (u, v) - \nabla \lambda (u, v) \|_{S^{- 1}} \leq
  \tfrac{2}{\sqrt{s}} \varepsilon . \nonumber
\end{align}
Then we deduce that
\begin{align}
  \varphi (u^+, v) - \varphi (u^{\star}, v^{\star}) ={} & \alpha (u^+) - \varphi
  (u^{\star}, v^{\star}) \nonumber\\
  ={} & \min_u \alpha (u) - \varphi (u^{\star}, v^{\star}) \nonumber\\
  \leq{} & \alpha (u - P^{- 1} \nabla \alpha (u)) - \varphi (u^{\star},
  v^{\star}) \nonumber\\
  ={} & \alpha (u - P^{- 1} \nabla \alpha (u)) - \ell (u - P^{- 1} \nabla \alpha
  (u)) + \ell (u - P^{- 1} \nabla \alpha (u)) - \varphi (u^{\star}, v^{\star})
  \nonumber\\
  ={} & \underbrace{\alpha (u - P^{- 1} \nabla \alpha (u)) - \ell (u - P^{- 1}
  \nabla \alpha (u))}_{\Delta_1} + \underbrace{\ell (u - P^{- 1} \nabla \alpha
  (u)) - \ell (u - P^{- 1} \nabla \ell (u))}_{\Delta_2} \nonumber\\
  & + \underbrace{\ell (u - P^{- 1} \nabla \ell (u)) - \varphi (u^{\star},
  v^{\star})}_{\Delta_3} . \nonumber
\end{align}

Now we bound $\Delta_1, \Delta_2$, and $\Delta_3$ respectively. For $\Delta_1$, if $\alpha (u - P^{- 1} \nabla \alpha (u)) \leq \alpha (u)$,
then $\varphi (u - P^{- 1} \nabla \alpha (u), v) \leq \varphi (u, v)$, and we
have
\[ \Delta_1 \leq \tfrac{168}{\sigma_2} ( \tfrac{1}{\sqrt{s}}\varepsilon^{3 / 2} +
   \tfrac{1}{s} \varepsilon^2 ). \]
To show $\alpha (u - P^{- 1} \nabla \alpha (u)) \leq \alpha (u)$, we first
note that $\nabla^2 \alpha (u) = \nabla^2_{u u} \varphi (u, v)$ and with
$\varepsilon \leq{\tfrac{s \sigma_2^2}{1296}}$, we have $4
\sqrt{\frac{\varepsilon}{s}} + \tfrac{2}{{s}} \varepsilon  \leq \tfrac{1}{2}$, which implies
\begin{align}
  P^{- 1 / 2} \nabla^2 \alpha (u) P^{- 1 / 2} ={} & P^{- 1 / 2} [\nabla^2 \ell
  (u) + \nabla^2 \alpha (u) - \nabla^2 \ell (u)] P^{- 1 / 2} \nonumber\\
  \preceq{} & I + \| P^{- 1 / 2} [\nabla^2 \alpha (u) - \nabla^2 \ell (u)] P^{-
  1 / 2} \| I \nonumber\\
  \preceq{} & I + \| S^{- 1 / 2} [\nabla^2 \varphi (u, v) - \nabla^2 \lambda (u,
  v)] S^{- 1 / 2} \| I \nonumber\\
  \preceq{} & \tfrac{3}{2} I \nonumber
\end{align}

and $\nabla^2 \alpha (u) \preceq \tfrac{3}{2} P$. Next, consider $\gamma
(\theta) \assign \alpha (u - \theta P^{- 1} \nabla \alpha (u))$ and let
$\theta_{\max} = \sup_{\theta}  \{ \theta > 0 : \gamma (\theta) = \alpha (u)
\}$. Since $\varepsilon > 0$ and $\nabla_v \varphi (u, v) = 0$, convexity of $\varphi$ implies $\nabla \alpha (u) \neq 0$. Moreover, $\alpha$ is coercive because $p_i > 0$ and $(A \mathe^{-v})_i > 0$ for every $i$. Thus $\theta_{\max}$ is finite and positive, since $- P^{- 1}
\nabla \alpha (u)$ is a descent direction of $\alpha$ at $u$.
By convexity, we have $\alpha (u - \theta P^{- 1} \nabla \alpha (u)) \leq
\alpha (u)$ for all $0 \leq \theta \leq \theta_{\max}$. Then we deduce that
\begin{align}
  \alpha (u) ={} & \gamma (\theta_{\max}) \nonumber\\
  ={} & \alpha (u) - \theta_{\max} \langle \nabla \alpha (u), P^{- 1} \nabla
  \alpha (u) \rangle \nonumber\\
  & + \theta_{\max}^2 \textstyle \int_0^1 \langle \nabla \alpha (u), P^{- 1} \nabla^2
  \alpha (u - t \theta_{\max} P^{- 1} \nabla \alpha (u)) P^{- 1} \nabla \alpha
  (u) \rangle (1 - t) \mathd t \nonumber\\
  \leq{} & \alpha (u) - \theta_{\max} \langle \nabla \alpha (u), P^{- 1} \nabla
  \alpha (u) \rangle + \theta_{\max}^2 \textstyle  \int_0^1 \langle \nabla \alpha
  (u), P^{- 1} ( \tfrac{3}{2} P ) P^{- 1} \nabla \alpha (u)
  \rangle (1 - t) \mathd t, \nonumber\\
  ={} & \alpha (u) - \theta_{\max} \| \nabla \alpha (u) \|_{P^{- 1}}^2 +
  \tfrac{3}{4} \theta_{\max}^2 \| \nabla \alpha (u) \|_{P^{- 1}}^2, \nonumber
\end{align}
where the inequality holds since $\alpha (u - t \theta_{\max} P^{- 1} \nabla
\alpha (u)) \leq \alpha (u)$. Given
\begin{equation}\label{eqn:steplb-contradiction}
	\tfrac{3}{4} \theta_{\max}^2 \| \nabla \alpha (u) \|_{P^{- 1}}^2 \geq
   \theta_{\max} \| \nabla \alpha (u) \|_{P^{- 1}}^2, 
\end{equation}
we have $\theta_{\max} \geq \tfrac{4}{3}$. Therefore, $\alpha (u - P^{- 1}
\nabla \alpha (u)) = \gamma (1) \leq \gamma (\theta_{\max}) = \alpha (u)$, and
it implies
\[ \Delta_1 \leq \tfrac{168}{\sigma_2} ( \tfrac{1}{\sqrt{s}}\varepsilon^{3 / 2} +
   \tfrac{1}{s} \varepsilon^2 ) . \]
For $\Delta_2$, since $\ell (u)$ is a quadratic function minimized at $u - P^{-1} \nabla \ell(u)$, we have
\begin{align}
  \Delta_2 ={} & \ell (u - P^{- 1} \nabla \alpha (u)) - \ell (u - P^{- 1} \nabla
  \ell (u)) \nonumber\\
  ={} & \tfrac{1}{2} \| \nabla \alpha (u) - \nabla \ell (u) \|_{P^{- 1}}^2
  \nonumber\\
  \leq{} & \tfrac{1}{2} \| \nabla \varphi (u, v) - \nabla \lambda (u, v)
  \|_{S^{- 1}}^2 \leq{} \tfrac{2}{s} \varepsilon^2 . \nonumber
\end{align}

For $\Delta_3$, we account for the residual $\nabla_v \lambda (u, v)$. As in the proof of \Cref{lem:contraction}, let $R = A^{\star}$. Completing the square gives
\begin{align}
  \lambda (u, v) - \varphi (\us, \vs) ={} & \tfrac{1}{2} \| u - \us \|^2_{P - R Q^{- 1} R^{\top}}
  + \tfrac{1}{2} \| \nabla_v \lambda (u, v) \|^2_{Q^{- 1}}, \nonumber\\
  \Delta_3 ={} & \tfrac{1}{2} \| v - \vs \|^2_{Q - R^{\top} P^{- 1} R} . \nonumber
\end{align}
Using
\begin{align}
  v - \vs ={} & Q^{- 1} R^{\top} (u - \us) + Q^{- 1} \nabla_v \lambda (u, v), \nonumber\\
  (P - R Q^{- 1} R^{\top}) (u - \us) ={} & \nabla_u \lambda (u, v) + R Q^{- 1} \nabla_v \lambda (u, v), \nonumber
\end{align}
and the spectral inequality established in the proof of \Cref{lem:contraction}, we deduce that
\begin{align}
  \Delta_3 \leq{} & (1 - \sigma_2) [\lambda (u, v) - \varphi (\us, \vs)] \nonumber\\
  & + \langle \nabla_u \lambda (u, v), P^{- 1} R Q^{- 1} \nabla_v \lambda (u, v) \rangle \nonumber\\
  & + \tfrac{1}{2} \| \nabla_v \lambda (u, v) \|^2_{\sigma_2 Q^{- 1} + Q^{- 1} R^{\top} P^{- 1} R Q^{- 1}} \nonumber\\
  \leq{} & (1 - \sigma_2) [\lambda (u, v) - \varphi (\us, \vs)]
  + \| \nabla_u \lambda (u, v) \|_{P^{- 1}} \| \nabla_v \lambda (u, v) \|_{Q^{- 1}} \nonumber\\
  & + \tfrac{1 + \sigma_2}{2} \| \nabla_v \lambda (u, v) \|^2_{Q^{- 1}}, \nonumber
\end{align}
where the last inequality uses Cauchy-Schwarz and again $Q -R^{\top} P^{- 1} R \succeq 0$. Since $\nabla_v \varphi (u, v) = 0$, \Cref{lem:grad-hess-new} implies
\begin{align}
  \| \nabla_v \lambda (u, v) \|_{Q^{- 1}} ={} & \| \nabla_v \lambda (u, v) - \nabla_v \varphi (u, v) \|_{Q^{- 1}}
  \leq \tfrac{2}{\sqrt{s}} \varepsilon, \nonumber\\
  \| \nabla_u \lambda (u, v) \|_{P^{- 1}} \leq{} & \| \nabla \varphi (u, v) \|_{S^{- 1}}
  + \| \nabla \lambda (u, v) - \nabla \varphi (u, v) \|_{S^{- 1}} \nonumber\\
  \leq{} & 2 \sqrt{\varepsilon} + \tfrac{4}{\sqrt{s}} \varepsilon
  \leq 6 \sqrt{\varepsilon}, \nonumber
\end{align}
where we use $\varepsilon \leq s$. Therefore, retaining the factor $1 - \sigma_2$ on the value-comparison error, we have
\begin{align}
  \Delta_3 \leq{} & (1 - \sigma_2) [\ell (u) - \varphi (u^{\star}, v^{\star})]
  + \tfrac{12}{\sqrt{s}} \varepsilon^{3 / 2} + \tfrac{4}{s} \varepsilon^2 \nonumber\\
  \leq{} & (1 - \sigma_2) \varepsilon + (1 - \sigma_2) | \ell (u) - \alpha (u) |
  + \tfrac{12}{\sqrt{s}} \varepsilon^{3 / 2} + \tfrac{4}{s} \varepsilon^2 \nonumber\\
  \leq{} & (1 - \sigma_2) \varepsilon + \tfrac{168 (1 - \sigma_2)}{\sigma_2}
  ( \tfrac{1}{\sqrt{s}}\varepsilon^{3 / 2} + \tfrac{1}{s} \varepsilon^2 )
  + \tfrac{12}{\sqrt{s}} \varepsilon^{3 / 2} + \tfrac{4}{s} \varepsilon^2 . \nonumber
\end{align}
Putting the relations together and using
\[ \tfrac{168 (2 - \sigma_2)}{\sigma_2} + 12 \leq \tfrac{336}{\sigma_2},
   \quad \tfrac{168 (2 - \sigma_2)}{\sigma_2} + 6 \leq \tfrac{338}{\sigma_2}, \]
we have
\begin{align}
  & \varphi (u^+, v) - \varphi (u^{\star}, v^{\star}) \nonumber\\
  \leq{} & (1 - \sigma_2) [\varphi (u, v) - \varphi (u^{\star}, v^{\star})] + 
\tfrac{336}{\sqrt{s}\sigma_2}  [\varphi (u, v) - \varphi (u^{\star},
  v^{\star})]^{3 / 2} + \tfrac{338}{s \sigma_2} [\varphi (u, v) - \varphi (u^{\star},
  v^{\star})]^2. \nonumber
\end{align}

Since $\nabla_u \varphi (u^+, v) = 0$ and $\varphi (u^+, v) \leq \varphi (u, v)$, the same argument with the roles of $u$ and $v$ interchanged applies to the next half-iteration. This completes the proof.

\subsection{Proof of Theorem \ref{thm:global-conv}}

For $\varepsilon \leq \tfrac{s \sigma_2^4}{1344^2}$, we have
\[ \tfrac{336}{\sigma_2 \sqrt{s}} [\varphi (u, v) - \varphi (u^{\star}, v^{\star})]^{3
   / 2} \leq \tfrac{\sigma_2}{4} [\varphi (u, v) - \varphi (u^{\star},
   v^{\star})] . \]
For $\varepsilon \leq \tfrac{s \sigma_2^2}{1352}$, we have
\[ \tfrac{338}{s \sigma_2} [\varphi (u, v) - \varphi (u^{\star}, v^{\star})]^2
   \leq \tfrac{\sigma_2}{4} [\varphi (u, v) - \varphi (u^{\star}, v^{\star})]
   . \]
Hence $\varphi (u^+, v) - \varphi (u^{\star}, v^{\star}) \leq ( 1 -
\tfrac{\sigma_2}{2} ) [\varphi (u, v) - \varphi (u^{\star},
v^{\star})]$, and it remains to bound
\begin{align}
  \varphi (u^1, v^1) - \varphi (u^{\star}, v^{\star}) ={} & \varphi ( \0_m,
  \0_n ) - \varphi (u^{\star}, v^{\star}) \nonumber\\
  \leq{} & \| \nabla \varphi ( \0_m, \0_n ) \|_1 \|
  ( \us, \vs ) \|_{\infty} \nonumber\\
  \leq{} & [ \| A \1_n - p \|_1 + \| A^{\top} \1_m - q
  \|_1 ] D \nonumber\\
  \leq{} & 2 [\| A \|_1 + \|p\|_1] D. \nonumber
\end{align}
This gives a total complexity of
\begin{align}
  K_{\varepsilon} \leq{} & \tfrac{2 \|p\|_1D^2}{\frac{\sigma_2^4}{1344^2}} + \tfrac{2}{\sigma_2} \log (
  \tfrac{2 [\| A \|_1 + \|p\|_1] D}{\varepsilon} ) \nonumber\\
  ={} & \tfrac{2 \cdot 1344^2 \|p\|_1D^2}{\sigma_2^4} + \tfrac{2}{\sigma_2} \log (2 [\| A \|_1 + \|p\|_1]
  D) + \tfrac{2}{\sigma_2} \log ( \tfrac{1}{\varepsilon} )
  . \nonumber
\end{align}
To ensure $\| \nabla \varphi (u, v) \| \leq \varepsilon'$, it suffices to have
$\| \nabla \varphi (u, v) \| \leq 4 \sqrt{\| p \|_1 \varepsilon} + 2
\varepsilon \leq \varepsilon'$, and it suffices to take
\[ \varepsilon \leq \min \{ \tfrac{(\varepsilon')^2}{64 \| p \|_1},
   \tfrac{\varepsilon'}{4} \} . \]
Since each step we analyze is a half-iteration, dividing the complexity by 2
completes the proof.

\subsection{Proof of Proposition \ref{thm: tight}}
Fix parameters $\theta\in(0,1)$ and $\rho_1,\rho_2, c_1, c_2 > 0$. Consider the matrix scaling problem with
\[
A=\left(\begin{smallmatrix}
\tfrac{1}{\rho_1 c_1} & \tfrac{\theta}{\rho_1 c_2}\\
\tfrac{\theta}{\rho_2 c_1} & \tfrac{1}{\rho_2 c_2}
\end{smallmatrix}\right),  \qquad p=q=\left(\begin{smallmatrix}1+\theta\\ 1+\theta\end{smallmatrix}\right),
\]
so that the solution is given by 
\[
A^\star=D_1^\star AD_2^\star =\left(\begin{smallmatrix}1&\theta\\ \theta&1\end{smallmatrix} \right)\quad 
\text{with} \quad D_1^\star  = \left(\begin{smallmatrix}
\rho_1 & 0\\
0 & \rho_2
\end{smallmatrix} \right), \qquad D_2^\star = \left(\begin{smallmatrix}
c_1 & 0\\
0 & c_2
\end{smallmatrix} \right).
\]

Let $M=P^{-1/2}A^\star Q^{-1}(A^\star)^\top P^{-1/2}$ and $L=I-M$. Since
$P=Q=(1+\theta)I$,
\[
  M=\tfrac{A^\star(A^\star)^\top}{(1+\theta)^2}
   =\tfrac{1}{(1+\theta)^2}\left(\begin{smallmatrix}1+\theta^2&2\theta\\ 2\theta&1+\theta^2\end{smallmatrix}\right).
\]
Its top eigenvector is $\1_2$ with eigenvalue $1$, and its second largest eigenvalue is
\[
  \mu_2=\operatorname{tr}(M)-1=\tfrac{2(1+\theta^2)-(1+\theta)^2}{(1+\theta)^2}
       =(\tfrac{1-\theta}{1+\theta})^{2}.
\]
Therefore,
\[
  \sigma_2(\theta)=\lambda_2(L)=1-\mu_2=\tfrac{4\theta}{(1+\theta)^2} .
\]

By the additive shift invariance of the dual potential, we can parametrize the deviation from the optimum by a single scalar $s$, taking $u(s)=u^\star+(s,-s)$ and
letting $v(s)$ be the exact minimizer
$v(s)=\arg\min_v\varphi(u(s),v)$ (a half Sinkhorn step). Next, starting at $A$ and $(u,v)$ is equivalent to starting at $A^\star$ and $(u-u^\star, v-v^\star)$,  with identical dual potential gaps and imbalance sequences. We adopt this change of variables, and substituting $v(s)$ into the dual potential gap
$G(s):=\varphi(u(s),v(s))-\varphi(\us, \vs)$ gives the closed form
\[
  G(s)=(1+\theta)\log\tfrac{(e^{s}+\theta e^{-s})(\theta e^{s}+e^{-s})}
       {(1+\theta)^2}
   =\tfrac{4\theta}{1+\theta} s^2+O(s^4)
\]
 In particular, $G$ is locally a positive-definite
quadratic in the imbalance $s$. Now we proceed to updating $u$ from $v(s)$ by
${u^+}=\log p- \log(A^\star {\mathe}^{-v(s)})$, and compute
$s^+:=\tfrac12\log({\mathe}^{u^+_1}/{\mathe}^{u^+_2})$. A direct differentiation of the
mapping $s \rightarrow s^+$ at $s=0$ gives, for a single half-step,
\[
  s^+=\big(\tfrac{1-\theta}{1+\theta}\big)^2s+\Ocal(s^2)=(1-\sigma_2)s+\Ocal(s^2),
\]
and substituting into $G(s)$, the dual potential gap $G_{k+1} = G(s^+)$ after one full Sinkhorn iteration is given by
\[
G_{k+1}=(1-\sigma_2)^2 G_k (1+o(1)).
\]
This matches \Cref{lem:global-conv} with equality, confirming that the local factor
 $(1-\sigma_2)^2$ per iteration is attained on this
family and cannot be improved. In conclusion, after a finite, $\varepsilon$-independent burn-in $c(\theta)$ (controlled
by \Cref{thm:global} and the threshold of \Cref{lem:global-conv}) the iterates enter the local regime,
where $\log(1/G_k)$ grows by at most $-2\log(1-\sigma_2)+o(1)=2\sigma_2+o(\sigma_2)$
per cycle. Hence reaching $G_k\le\varepsilon$ requires
\[
  k\;\ge\;c(\theta)+\tfrac{\log(1/\varepsilon)}{-2\log(1-\sigma_2)}
        \;=\;c(\theta)+\tfrac{1}{2\sigma_2(\theta)}\log(\tfrac1\varepsilon)(1+o(1)),
\]
matching the upper bound $\Ocal\big(c+\tfrac{1}{\sigma_2}\log(\tfrac1\varepsilon)\big)$ of
\Cref{thm:global-conv} up to the additive constant.

\section{Proofs of results in Section \ref{sec:acc}}

\subsection{Auxiliary results}

\begin{lem} \label{lem:grad-small}
  Suppose $f$ is $L$-smooth and convex. Then there exists an algorithm that
  outputs $x^K$ such that $\| \nabla f (x^K) \|^2 \leq \varepsilon$ in at most
  $K_{\varepsilon} \assign \big\lceil ( \tfrac{2112 L^2 \| x^1 -
  x^{\star} \|^2}{\varepsilon} )^{1 / 4} \big\rceil$ iterations, where
  $x^{\star}$ is an optimal solution.
\end{lem}

\begin{proof}
  The algorithm is \texttt{FISTA} + \texttt{FISTA-G}. For $K_{\varepsilon} \geq 2$, run $\lfloor K_{\varepsilon} / 2 \rfloor$ iterations of each phase. Corollary 1 from {\cite{lee2021geometric}}, specialized to the smooth case, bounds the squared gradient norm of the output by
  \[ \tfrac{132 L^2 \| x^1 - x^{\star} \|^2}{(\lfloor K_{\varepsilon} / 2 \rfloor + 2)^4}
     \leq \tfrac{2112 L^2 \| x^1 - x^{\star} \|^2}{K_{\varepsilon}^4} \leq \varepsilon . \]
  If $K_{\varepsilon} \leq 1$, the initial point already satisfies $\| \nabla f (x^1) \|^2 \leq L^2 \| x^1 - x^{\star} \|^2 \leq \varepsilon$. This completes the proof.
\end{proof}

\begin{lem} \label{lem:katyusha}
  Suppose $f$ is the sum of $\frac{1}{n} \sum_{j = 1}^n f_j$ and a $\sigma$-strongly convex function with a computable proximal operator, with each $f_j$ convex and
  $L_{\max}$-smooth. Then there exists an algorithm
  that outputs $x^K$ such that $\mathbb{E} [f (x^K)] - f (x^{\star}) \leq
  \varepsilon$ in $K_{\varepsilon} \assign \mathcal{O} \big( \big(n + \sqrt{\tfrac{n
     L_{\max}}{\sigma}}\big) \log ( 2 + \tfrac{f (x^1) - f (x^{\star})}{\varepsilon}
     ) \big)$ stochastic gradient oracle calls.
\end{lem}

\begin{proof}
  The algorithm is \texttt{Katyusha} \cite{allen2018katyusha}. Invoking Theorem 2.1 of
  {\cite{allen2018katyusha}} completes the proof.
\end{proof}

\subsection{Proof of Lemma \ref{lem:reduced-pot}}
Let $a_{[:, j]}$ denote the $j$-th column of $A$. Through direct calculation, we have
\begin{align}
  \zeta (u) ={} & \textstyle \varphi (u, - \log (q / A^{\top} \mathe^u)) \nonumber\\
  ={} & \textstyle \sum_{j = 1}^n q_j \log \langle a_{[:, j]}, \mathe^u \rangle - \langle
  p, u \rangle - \sum_{j = 1}^n q_j \log q_j + \langle \1_n, q
  \rangle \nonumber
\end{align}
and that
\begin{align}
  \nabla \zeta (u) ={} & \textstyle \sum_{j = 1}^n \tfrac{q_j}{\langle a_{[:, j]}, \mathe^u
  \rangle} \mathe^U a_{[:, j]} - p \nonumber\\
  \nabla^2 \zeta (u) ={} & \textstyle \mathcal{D} ( \sum_{j = 1}^n q_j \tfrac{\mathe^U
  a_{[:, j]}}{\langle a_{[:, j]}, \mathe^u \rangle} ) - \sum_{j = 1}^n
  q_j \tfrac{\mathe^U a_{[:, j]} a_{[:, j]}^{\top} \mathe^U}{\langle a_{[:,
  j]}, \mathe^u \rangle^2} \nonumber\\
  ={} & \textstyle \mathcal{D} ( \sum_{j = 1}^n q_j \tfrac{\mathe^U a_{[:,
  j]}}{\langle \1_m, \mathe^U a_{[:, j]} \rangle} ) - \sum_{j
  ={} 1}^n q_j \tfrac{\mathe^U a_{[:, j]} a_{[:, j]}^{\top}
  \mathe^U}{\langle \1_m, \mathe^U a_{[:, j]} \rangle^2}
  \nonumber\\
  \backassign & \textstyle \sum_{j = 1}^n q_j [\mathcal{D} (\sigma_j) - \sigma_j
  \sigma_j^{\top}], \nonumber
\end{align}
where we define $\sigma_j \assign \tfrac{\mathe^U a_{[:, j]}}{\langle
\1_m, \mathe^U a_{[:, j]} \rangle}$ and $ \0_m \leq \sigma_j \leq \1_m$ by
nonnegativity of $A$ and $\mathe^U$, and the fact that $A$ must have at least one nonzero entry per column. To establish smoothness, we have
\[ \textstyle \| \nabla^2 \zeta (u) \| \leq \sum_{j = 1}^n q_j \| \mathcal{D} (\sigma_j)
   - \sigma_j \sigma_j^{\top} \| \leq \tfrac{\| q \|_1}{2} = \tfrac{\| p
   \|_1}{2}, \]
where we use Lipschitzness of softmax functions \cite{nair2026softmax}: $\| \mathcal{D} (x) - x
x^{\top} \| \leq \tfrac{1}{2}$ for $x \in \{ x \geq \0_m : \langle \1_m,
x \rangle = 1 \}$. To establish Lipschitz continuity of the Hessian, we similarly deduce, for any $x, y$, that
\begin{align}
  \| \nabla^2 \zeta (x) - \nabla^2 \zeta (y) \| ={} & \textstyle \| \sum_{j = 1}^n q_j
  [\mathcal{D} (\sigma_j) - \sigma_j \sigma_j^{\top}] - \sum_{j = 1}^n q_j
  [\mathcal{D} (\sigma_j') - \sigma_j' (\sigma_j')^{\top}] \|
  \nonumber\\
  \leq{} & \textstyle \sum_{j = 1}^n q_j \| \sigma_j - \sigma_j' \|_{\infty} + \sum_{j =
  1}^n q_j \| \sigma_j \sigma_j^{\top} - \sigma_j' (\sigma_j')^{\top} \|
  \nonumber\\
  ={} & \textstyle \sum_{j = 1}^n q_j \| \sigma_j - \sigma_j' \|_{\infty} + \sum_{j = 1}^n
  q_j \| \sigma_j \sigma_j^{\top} - \sigma_j (\sigma_j')^{\top} + \sigma_j
  (\sigma_j')^{\top} - \sigma_j' (\sigma_j')^{\top} \| \nonumber\\
  \leq{} & \textstyle \sum_{j = 1}^n q_j \| \sigma_j - \sigma_j' \|_{\infty} + \sum_{j =
  1}^n q_j [\| \sigma_j \| + \| \sigma_j' \|] \| \sigma_j - \sigma_j' \|
  \nonumber\\
  \leq{} & \textstyle 3 \sum_{j = 1}^n q_j \| \sigma_j - \sigma_j' \|, \label{eqn:property-reduced-pot-1}
\end{align}
where $\sigma_j = \tfrac{\mathe^X a_{[:, j]}}{\langle
\1_m, \mathe^X a_{[:, j]} \rangle}$ and $\sigma_j' = \tfrac{\mathe^Y a_{[:, j]}}{\langle
\1_m, \mathe^Y a_{[:, j]} \rangle}$ are both outcomes of weighted softmax; \eqref{eqn:property-reduced-pot-1} uses the fact that $\|\sigma_j\| \leq 1$, 
 and
Lipschitzness of softmax  implies 
\[ \| \sigma_j - \sigma_j' \| \leq \tfrac{1}{2} \| x - y \| . \]
Plugging back, we have $\| \nabla^2 \zeta (x) - \nabla^2 \zeta
(y) \| \leq \frac{3}{2} \sum_{j = 1}^n q_j \| x - y \| = \tfrac{3}{2} \| p
\|_1 \| x - y \|$. Finally, the quadratic growth 
$\tfrac{\sigma_2}{4} \| \Pi ( u - \us ) \|^2_P\leq \zeta (u) - \zeta (u^{\star})$
follows from the symmetric argument in \Cref{rem:symm}, since $\Pi (u - \us)$ is obtained by shifting $\us$ along $\1_m$ so that $u - \us \perp p$. To show the PL inequality, taking $u^\star$ such that $\Pi(u-\us) = u-\us$, we deduce that
\begin{align}
  \zeta (u) - \zeta (u^{\star}) \leq{} & \langle \nabla \zeta (u), u - \us
  \rangle   \label{eqn:property-reduced-pot-2}\\
  ={} & \langle P^{- 1 / 2} \nabla \zeta (u), P^{1 / 2}  ( u - \us
  ) \rangle \nonumber\\
  \leq{} & \| \nabla \zeta (u) \|_{P^{- 1}} \| \Pi( u - \us)
  \|_P \label{eqn:property-reduced-pot-4}\\
  \leq{}{} & \| \nabla \zeta (u) \|_{P^{- 1}} \sqrt{\tfrac{4}{\sigma_2} [\zeta (u)
  - \zeta (u^{\star})]},   \label{eqn:property-reduced-pot-5}
\end{align}
where \eqref{eqn:property-reduced-pot-2} uses convexity of $\zeta$; \eqref{eqn:property-reduced-pot-4} uses Cauchy-Schwarz and \eqref{eqn:property-reduced-pot-5} uses quadratic growth. Squaring and dividing both sides by $\zeta (u) - \zeta (u^{\star})$ gives the desired relation.
 Finally, with \Cref{lem:grad-hess-new}, we have
\[ \| \nabla \zeta (u) \|_{P^{- 1}}^2 = \| \nabla \varphi (u, v (u)) \|_{S^{-
   1}}^2 \leq 4 \varepsilon + \tfrac{8}{\sqrt{s}} \varepsilon^{3 / 2} +
   \tfrac{4}{s} \varepsilon^2 \]
and with $\varepsilon \leq s$, we have $\tfrac{8}{\sqrt{s}} \varepsilon^{3 /
2} \leq 8 \varepsilon$ and $\tfrac{4}{s} \varepsilon^2 \leq 4 \varepsilon$,
giving $\| \nabla \zeta (u) \|_{P^{- 1}}^2 \leq 16 \varepsilon = 16 [\zeta (u)
- \zeta (u^{\star})]$. To establish lower and upper bounds on the spectrum, we use invariance under shifts along $\1_m$, $\| \Pi ( u - \us ) \| \leq \tfrac{1}{\sqrt{s}} \| \Pi ( u - \us ) \|_P$, and $\| \Pi ( u - \us )
\|_P \leq \sqrt{\tfrac{4 \varepsilon}{\sigma_2}}$. By Weyl's
inequality,
\[ \lambda_2 (P^{- 1 / 2} \nabla^2 \zeta (u) P^{- 1 / 2}) \geq \lambda_2
   ( P^{- 1 / 2} \nabla^2 \zeta ( \us ) P^{- 1 / 2} ) -
   \tfrac{H}{s^{3 / 2}} \| \Pi ( u - \us ) \|_P \geq \sigma_2 -
   \sqrt{\tfrac{4 H^2 \varepsilon}{s^3 \sigma_2}} . \]
\[ \lambda_m (P^{- 1 / 2} \nabla^2 \zeta (u) P^{- 1 / 2}) \leq \lambda_m
   ( P^{- 1 / 2} \nabla^2 \zeta ( \us ) P^{- 1 / 2} ) +
   \tfrac{H}{s^{3 / 2}} \| \Pi ( u - \us ) \|_P \leq 1 +
   \sqrt{\tfrac{4 H^2 \varepsilon}{s^3 \sigma_2}} . \]

Hence for $\varepsilon \leq \tfrac{\sigma_2^3 s^3}{16 H^2} = \tfrac{\sigma_2^3 s^3}{36 \| p \|_1^2}$, $\sigma_2 -
\sqrt{\tfrac{4 H^2 \varepsilon}{s^3 \sigma_2}} \geq \tfrac{\sigma_2}{2}$ and $1 + \sqrt{\tfrac{4 H^2 \varepsilon}{s^3 \sigma_2}} \leq 1 + \tfrac{\sigma_2}{2} \leq 2$. This completes the proof.

\subsection{Proof of Theorem \ref{thm:global-acc}}

Initialize at $u^1 = \0_m$ and choose $(\us, \vs)$ to attain $D$ in \eqref{eqn:diam}. Given that $\zeta$ is $L$-smooth and convex with $L = \tfrac{1}{2} \| p \|_1$,
invoking \Cref{lem:grad-small} with squared-gradient tolerance $\varepsilon^2$ gives a sufficient iteration count bounded by
\[ \Big\lceil (\tfrac{528 \| p \|_1^2 \| u^{\star} \|^2}{\varepsilon^2})^{1 / 4} \Big\rceil
   \leq \Big\lceil \tfrac{8 \sqrt{\| p \|_1 \| u^{\star} \|}}{\sqrt{\varepsilon}} \Big\rceil . \]
The second inequality in the theorem uses $\| u^{\star} \| \leq \sqrt{m} \| u^{\star}
\|_{\infty} \leq \sqrt{m} D$. Computing $\nabla \zeta (u)$ requires $A^{\top} \mathe^u$ to obtain $v = - \log (q / A^{\top}
\mathe^u)$ and then $A \mathe^{-v}$ to obtain the row residual. These two matrix-vector products cost $\mathcal{O} (\tmop{nnz} (A))$, the same order as one full {\sk} iteration. The column residual is zero, so $\| \nabla \zeta (u) \| \leq \varepsilon$ gives an $\varepsilon$-approximate scaling.

\subsection{Proof of Theorem \ref{thm:global-vr}}
Initialize at $u^1 = \0_m$ and choose $(\us, \vs)$ to attain $D$ in \eqref{eqn:diam}. If $\| \nabla \zeta (\0_m) \| \leq \varepsilon$, return $\hat{u} = \0_m$. Otherwise, $\| u^{\star} \| > 0$, and we adopt Nesterov's regularization technique for making the gradient small \cite{nesterov2013introductory}. Define
\[ \zeta_{\sigma} (u) \assign \zeta (u) + \tfrac{\sigma}{2} \| u \|^2 . \]
Since $\zeta$ is $L$-smooth and convex, $\zeta_{\sigma}$ is $(L +
\sigma)$-smooth and $\sigma$-strongly convex. Denote $u^{\star}_{\sigma} =
\arg \min_u \zeta_{\sigma} (u)$. By the optimality condition, $\nabla \zeta_{\sigma}
(u_{\sigma}^{\star}) = \nabla \zeta (u_{\sigma}^{\star}) + \sigma
u_{\sigma}^{\star} = 0$ and with quadratic growth,
\[ \zeta_{\sigma} (u^{\star}) - \zeta_{\sigma} (u_{\sigma}^{\star}) \geq
   \tfrac{\sigma}{2} \| u_{\sigma}^{\star} - \us \|^2 . \]
By definition, the above relation implies
\[ \tfrac{\sigma}{2} \| u_{\sigma}^{\star} - \us \|^2 \leq
   \zeta_{\sigma} (u^{\star}) - \zeta_{\sigma} (u_{\sigma}^{\star}) = \zeta
   (u^{\star}) + \tfrac{\sigma}{2} \| u^{\star} \|^2 - \zeta
   (u^{\star}_{\sigma}) - \tfrac{\sigma}{2} \| u^{\star}_{\sigma} \|^2 \leq
   \tfrac{\sigma}{2} \| u^{\star} \|^2 - \tfrac{\sigma}{2} \|
   u^{\star}_{\sigma} \|^2 \]
since $\zeta(u^\star) \leq \zeta(u^\star_\sigma)$. Rearranging, we have $\| u^{\star}_{\sigma} \|^2 \leq \| u_{\sigma}^{\star} - \us
\|^2 + \| u^{\star}_{\sigma} \|^2 \leq \| u^{\star} \|^2$ and 
$\| \nabla \zeta (u_{\sigma}^{\star}) \| = \sigma \| u_{\sigma}^{\star} \|
\leq \sigma \| u^{\star} \|$. To apply \Cref{lem:katyusha}, write
\begin{align}
  \zeta_{\sigma} (u) ={} & \tfrac{1}{n} \textstyle \sum_{j = 1}^n [n q_j \log \langle a_{[:, j]}, \mathe^u \rangle] \nonumber\\
  & + \tfrac{\sigma}{2} \| u \|^2 - \langle p, u \rangle - \textstyle \sum_{j = 1}^n q_j \log q_j + \langle \1_n, q \rangle . \nonumber
\end{align}
Each smooth component is convex and $\tfrac{n q_j}{2}$-smooth, so $L_{\max} = \tfrac{n}{2} \| q \|_\infty$. The remaining terms form a separable $\sigma$-strongly convex function with a closed-form proximal operator. Now suppose we run \texttt{Katyusha} on $\zeta_{\sigma}$, which, by
\Cref{lem:katyusha}, outputs $\hat{u}$ such that
\[ \mathbb{E} [\zeta_{\sigma} (\hat{u})] - \zeta_{\sigma} (u^{\star}_{\sigma})
   \leq \hat{\varepsilon} \]
in $\mathcal{O} \big( \big( n + \sqrt{\tfrac{n L_{\max}}{\sigma}} ) \log
( 2 + \tfrac{\zeta_{\sigma} (u^1) - \zeta_{\sigma}
(u^{\star}_{\sigma})}{\hat{\varepsilon}} \big) \big)$ stochastic gradient
calls. Now we deduce that
\begin{align}
  \mathbb{E} [\| \nabla \zeta (\hat{u}) \|] ={} & \mathbb{E} [\| \nabla \zeta
  (\hat{u}) - \nabla \zeta (u_{\sigma}^{\star}) + \nabla \zeta
  (u_{\sigma}^{\star}) \|] \nonumber\\
  \leq{} & \| \nabla \zeta (u_{\sigma}^{\star}) \| + \mathbb{E} [\| \nabla \zeta (\hat{u}) - \nabla \zeta
  (u_{\sigma}^{\star}) \|] 
  \label{eqn:katyusha-1}\\
  \leq{} & \sigma \| u^{\star} \| + L\mathbb{E} [\| \hat{u} - u_{\sigma}^{\star}
  \|] \label{eqn:katyusha-2}\\
  \leq{} & \sigma \| u^{\star} \| + L \sqrt{\mathbb{E} [\| \hat{u} -
  u_{\sigma}^{\star} \|^2]} \label{eqn:katyusha-3}\\
  \leq{} & \sigma \| u^{\star} \| + L \sqrt{\tfrac{2}{\sigma} \mathbb{E}
  [\zeta_{\sigma} (\hat{u}) - \zeta_{\sigma} (u^{\star}_{\sigma})]}
  \label{eqn:katyusha-4}\\
  \leq{} & \sigma \| u^{\star} \| + L \sqrt{\tfrac{2}{\sigma} \hat{\varepsilon}}, \nonumber
\end{align}
where \eqref{eqn:katyusha-1} uses the triangle inequality; \eqref{eqn:katyusha-2} uses the previous relation $\| \nabla \zeta (u_{\sigma}^{\star}) \| \leq \sigma \| u^{\star} \|$ and $L$-smoothness; \eqref{eqn:katyusha-3} uses Jensen's inequality $\Ebb[|X|] \leq \sqrt{\Ebb[X^2]}$; \eqref{eqn:katyusha-4} uses quadratic growth. Taking $\sigma = \tfrac{\varepsilon}{2 \| u^{\star} \|}$ and
$\hat{\varepsilon} = \tfrac{\sigma \varepsilon^2}{8 L^2} = \tfrac{\varepsilon^3}{16 L^2 \|
u^{\star} \|}$, we have
\[ \mathbb{E} [\| \nabla \zeta (\hat{u}) \|] \leq \tfrac{\varepsilon}{2 \|
   u^{\star} \|} \| u^{\star} \| + L \sqrt{\tfrac{2}{\sigma} \tfrac{\sigma \varepsilon^2}{8
   L^2}} = \tfrac{\varepsilon}{2} + \tfrac{\varepsilon}{2} = \varepsilon,\]
which gives a total complexity of
\begin{align}
  \mathcal{O} \big( \big( n + \sqrt{\tfrac{n L_{\max}}{\sigma}} \big) \log \big( 2 +
  \tfrac{\zeta_{\sigma} (u^1) - \zeta_{\sigma}
  (u^{\star}_{\sigma})}{\hat{\varepsilon}} \big) \big) ={} & \mathcal{O}
  \big( \big( n + \sqrt{\tfrac{2 n L_{\max} \| u^{\star} \|}{\varepsilon}} \big)
  \log \big( 2 + \tfrac{16 L^2 \| u^{\star} \| [\zeta_{\sigma} (u^1) -
  \zeta_{\sigma} (u^{\star}_{\sigma})]}{\varepsilon^3} \big) \big)
  \nonumber\\
  ={} & \mathcal{O} \big( \big( n + n^{1 / 2} m^{1 / 4} \sqrt{\tfrac{2 L_{\max} \| u^{\star}
  \|_{\infty}}{\varepsilon}} \big) \log \big( 2 + \tfrac{16 L^2 \| u^{\star} \|
  [\zeta_{\sigma} (u^1) - \zeta_{\sigma} (u^{\star}_{\sigma})]}{\varepsilon^3}
  \big) \big) \nonumber\\
  ={} & \tilde{\mathcal{O}} \big( n + n^{1 / 2} m^{1 / 4} \sqrt{\tfrac{2 L_{\max} \| u^{\star}
  \|_{\infty}}{\varepsilon}} \big), \nonumber
\end{align}

where we use $\| u^{\star} \| \leq \sqrt{m} \| u^{\star} \|_{\infty}$ and hide the logarithmic terms. Since $u^1 = \0_m$ and the column-normalized matrix has total mass $\|q\|_1 = \|p\|_1$, we have $\| \nabla \zeta (\0_m) \|_1 \leq 2 \| p \|_1$ and
\[ \zeta_{\sigma} (u^1) - \zeta_{\sigma} (u^{\star}_{\sigma}) = \zeta (
   \0_m ) - \zeta_{\sigma} (u^{\star}_{\sigma}) \leq \zeta ( \0_m
   ) - \zeta (u^{\star}_{\sigma}) \leq \zeta ( \0_m ) - \zeta
   (u^{\star}) \leq \| u^{\star} \|_{\infty}  \| \nabla \zeta
   ( \0_m ) \|_1 \leq 2 \| p \|_1 \| u^{\star} \|_{\infty} . \]
Plugging in $L_{\max} = \tfrac{n}{2} \| q \|_\infty$ gives $\tilde{\mathcal{O}} ( n +
n m^{1 / 4} \sqrt{\tfrac{\| q \|_\infty \| u^{\star} \|_{\infty}}{\varepsilon}}
)$ complexity of stochastic gradients. Each smooth component depends only on the nonzero entries of its column. The remaining affine and quadratic terms are handled by the separable proximal update. Using deferred coordinate updates as in \cite[Remark 2.2]{allen2018katyusha}, uniform column sampling gives expected amortized cost $\mathcal{O} ( \tfrac{\tmop{nnz} (A)}{n}
)$ per iteration, including the full-gradient computation every $\mathcal{O} (n)$ iterations. Thus the expected total arithmetic complexity is
\[ \tilde{\mathcal{O}} ( \tmop{nnz} (A) + {\tmop{nnz} (A)}\cdot m^{1 /
   4} \sqrt{\tfrac{ \| q \|_\infty \| u^{\star} \|_{\infty}}{\varepsilon}}
   ) . \]
The regularization parameter can also be chosen without knowing $\| u^{\star} \|$. Start with $\sigma = L$ and, for each trial, run \texttt{Katyusha} from $\0_m$ with fresh randomness and expected objective tolerance $\hat{\varepsilon} = \tfrac{\sigma \varepsilon^2}{32 L^2}$. The computable bound
\[ \zeta_{\sigma} (\0_m) - \zeta_{\sigma} (u^{\star}_{\sigma}) \leq \tfrac{\| \nabla \zeta (\0_m) \|^2}{2 \sigma} \]
determines a sufficient inner iteration count in \Cref{lem:katyusha}. Compute the full gradient at the returned point and stop if $\| \nabla \zeta (\hat{u}) \| \leq \varepsilon$; otherwise halve $\sigma$ and repeat. Once $\sigma \leq \tfrac{\varepsilon}{4 \| u^{\star} \|}$, the preceding gradient estimate gives $\mathbb{E} [\| \nabla \zeta (\hat{u}) \|] \leq \varepsilon / 2$, so Markov's inequality bounds the probability of another unsuccessful trial by $1 / 2$. The work per trial grows by at most a factor $\sqrt{2}$ per halving, apart from logarithmic factors. The trials before this threshold add only logarithmic overhead, and the expected work after it is summable since $\sqrt{2} / 2 < 1$. Thus this procedure has the same expected arithmetic complexity, including the full-gradient checks. Noting that $\|\us\|_\infty \leq D$, this completes the proof.

\subsection{Proof of Lemma \ref{lem:hu-property}}

Convexity of $h_u$ follows from the composition rule of convex functions. To
show smoothness, it suffices to show that $P^{- 1 / 2} \nabla^2 h_u (w) P^{- 1
/ 2} \preceq L \cdot I$, which holds since
\[ \nabla^2 h_u (w) =\mathcal{D} \big( \tfrac{\nabla \zeta (u)}{\| P^{- 1 /
   2} \nabla \zeta (u) \|} \big) \nabla^2 \zeta (u -\mathcal{D} (w) \nabla
   \zeta (u)) \mathcal{D} \big( \tfrac{\nabla \zeta (u)}{\| P^{- 1 / 2}
   \nabla \zeta (u) \|} \big) \]
and for some $\hat{w}$, we have 
\begin{align}
  P^{- 1 / 2} \nabla^2 h_u (w) P^{- 1 / 2} ={} & \mathcal{D} \big( \tfrac{P^{-
  1 / 2} \nabla \zeta (u)}{\| P^{- 1 / 2} \nabla \zeta (u) \|} \big)
  \nabla^2 \zeta (\hat{w}) \mathcal{D} \big( \tfrac{P^{- 1 / 2} \nabla \zeta
  (u)}{\| P^{- 1 / 2} \nabla \zeta (u) \|} \big) \preceq{} L \cdot I. \nonumber
\end{align}

Given $\zeta (u) - \zeta ( \us ) \leq  
\tfrac{s \sigma_2^2}{1296}$, by \Cref{lem:grad-hess-new}, $\| \nabla
\zeta (u) \|_{P^{- 1}} = \| \nabla \varphi (u, v (u)) \|_{S^{- 1}} \leq 2
\sqrt{\varepsilon} + \tfrac{2}{\sqrt{s}} \varepsilon$ and

\begin{align}
  P^{- 1 / 2} \nabla^2 \zeta (u) P^{- 1 / 2} \preceq{} & P^{- 1 / 2}
  (\mathcal{D} (\nabla \zeta (u) + p)) P^{- 1 / 2} \nonumber\\
  \preceq{} & \mathcal{D} (P^{- 1}
  \nabla \zeta (u)) + I \nonumber\\
  \preceq{} & \tfrac{1}{\sqrt{s}} \| \nabla \zeta (u) \|_{P^{- 1}} I + I
  \nonumber\\
  \preceq{} & ( \tfrac{2}{\sqrt{s}} \sqrt{\varepsilon} + \tfrac{2}{s}
  \varepsilon + 1 ) I \preceq{} \tfrac{3}{2} I . \nonumber
\end{align}

Denote $u_t \assign u - t P^{- 1} \nabla \zeta (u)$. With the same reasoning
as \eqref{eqn:steplb-contradiction}, we can show that $\zeta (u - P^{- 1} \nabla \zeta (u)) \leq \zeta (u)$
and $\zeta (u_t) - \zeta ( \us ) \leq \varepsilon$ for all $t \in [0, 1]$,
giving $\nabla^2 \zeta (u_t) \preceq \tfrac{3}{2} P$. Then we deduce that
\begin{align}
\| \nabla \zeta (u) \|^2_{P^{- 1}} h_u ( P^{-1} \1_m )
  ={} & \zeta (u - P^{- 1} \nabla \zeta (u)) - \zeta (u) \nonumber\\
  ={} & \textstyle - \langle \nabla \zeta (u), P^{- 1} \nabla \zeta (u) \rangle + \int_0^1
  (1 - t)  \langle P^{- 1} \nabla \zeta (u), \nabla^2 \zeta (u_t) P^{- 1}
  \nabla \zeta (u) \rangle \mathd t \nonumber\\
  \leq{} & \textstyle - \| \nabla \zeta (u) \|^2_{P^{- 1}} + \int_0^1 (1 - t)  \langle
  P^{- 1} \nabla \zeta (u), ( \tfrac{3}{2} P ) P^{- 1} \nabla \zeta
  (u) \rangle \mathd t \nonumber\\
  ={} & - \| \nabla \zeta (u) \|^2_{P^{- 1}} + \tfrac{3}{4} \| \nabla \zeta (u)
  \|^2_{P^{- 1}} = - \tfrac{1}{4} \| \nabla \zeta (u) \|^2_{P^{- 1}} .
  \nonumber
\end{align}

Dividing both sides by $\| \nabla \zeta (u) \|^2_{P^{- 1}}$ completes the
proof.

\subsection{Proof of Lemma \ref{lem:pot}}

A good scaling vector $w$ makes $h_u
(w)$ small. Since $h_u$ is convex, its negative gradient aligns with any direction that points to a scaling vector better than the current $w$, say $\hat{w}$:
\[ h_u (\hat{w}) \geq h_u (w) + \langle \nabla h_u (w), \hat{w} - w
   \rangle \quad \Rightarrow \quad \langle - P^{- 1 / 2} \nabla h_u (w),
   P^{1 / 2} (\hat{w} - w) \rangle \geq h_u (w) - h_u (\hat{w}) \geq 0 . \]
Hence, preconditioned gradient descent on $w$ reduces $\tfrac{1}{2} \| w -
\hat{w} \|^2_P$: define $w^+ = w - \frac{1}{L} P^{- 1} \nabla h_u (w)$. We have
\begin{align}
  \tfrac{1}{2} \| w^+ - \hat{w} \|^2_P \leq{} & \tfrac{1}{2} \| w - \hat{w}
  \|^2_P - \tfrac{1}{L} [h_u (w) - h_u (\hat{w})] + \tfrac{1}{2L^2} \| \nabla h_u
  (w) \|^2_{P^{- 1}} \nonumber\\
  \leq{} & \tfrac{1}{2} \| w - \hat{w} \|^2_P - \tfrac{1}{L} \underbrace{[h_u (w^+) - h_u
  (\hat{w})]}_{\Delta h}, \label{eqn:ogd-inequality}
\end{align}\vspace{-5pt}

where $ h_u (w^+) \leq  h_u (w) - \tfrac{1}{2 L} \| \nabla h_u (w)
\|^2_{P^{- 1}}$ by $L$-smoothness in \Cref{lem:hu-property}. Next, we have, by definition, that
\begin{align}
  \log (\zeta (u^+) - \zeta (u^{\star})) - \log (\zeta (u) - \zeta
  (u^{\star})) ={} & \log ( \tfrac{\zeta (u^+) - \zeta (u^{\star})}{\zeta
  (u) - \zeta (u^{\star})} ) \nonumber\\
  ={} & \log ( \tfrac{\zeta (u^+) - \zeta (u) + \zeta (u) - \zeta
  (u^{\star})}{\zeta (u) - \zeta (u^{\star})} ) \nonumber\\
  ={} & \log ( 1 + \tfrac{\zeta (u^+) - \zeta (u)}{\zeta (u) - \zeta
  (u^{\star})} ) \nonumber\\
  ={} & \log ( 1 + \tfrac{\zeta (u^+) - \zeta (u)}{\| \nabla \zeta (u)
  \|^2_{P^{- 1}}} \tfrac{\| \nabla \zeta (u) \|^2_{P^{- 1}}}{\zeta (u) - \zeta
  (u^{\star})} ) \nonumber \\
  \leq{} & \log ( 1 + \tfrac{\sigma_2}{4} \min \{ h_u (w^+), 0 \} ) \label{eqn:pot-red-1}
  \\
  \leq{} & \log ( 1 + \tfrac{\sigma_2}{4} h_u (w^+) ) \label{eqn:pot-red-2} \\
  \leq{} & \tfrac{\sigma_2}{4} h_u (w^+), \label{eqn:pot-red-3}
\end{align}
where \eqref{eqn:pot-red-1} uses the PL inequality from \Cref{lem:reduced-pot} and the fact that $\zeta(u^+) \leq \min\{\zeta(u), \zeta(u - \Dcal(w^+) \nabla \zeta(u) )\}$; \eqref{eqn:pot-red-2} uses monotonicity of $\log$ and \eqref{eqn:pot-red-3} uses $\log (1 + x) \leq x$.
Together with $\tfrac{L \sigma_2}{8} \| w^+ - \hat{w} \|^2_P - \tfrac{L
\sigma_2}{8} \| w - \hat{w} \|^2_P \leq - \tfrac{\sigma_2}{4} [h_u (w^+) -
h_u (\hat{w})]$ from \eqref{eqn:ogd-inequality}, we have
\begin{align}
  \Omega _{\hat{w}}(u^+, w^+) ={} & \log (\zeta (u^+) - \zeta (u^{\star})) + \tfrac{L
  \sigma_2}{8} \| w^+ - \hat{w} \|^2_P \nonumber\\
  \leq{} & \log (\zeta (u) - \zeta (u^{\star})) + \tfrac{L \sigma_2}{8} \| w -
  \hat{w} \|^2_P - \tfrac{\sigma_2}{4} [h_u (w^+) - h_u (\hat{w})] +
  \tfrac{\sigma_2}{4} h_u (w^+) \nonumber\\
  ={} & \log (\zeta (u) - \zeta (u^{\star})) + \tfrac{L \sigma_2}{8} \| w -
  \hat{w} \|^2_P + \tfrac{\sigma_2}{4} h_u (\hat{w}) \nonumber\\
  ={} & \Omega _{\hat{w}}(u, w) + \tfrac{\sigma_2}{4} h_u (\hat{w}),\nonumber
\end{align}
and this completes the proof.

\subsection{Proof of Lemma \ref{lem:local-contraction}}

For brevity, we drop the iteration index and use $u, u^{1 / 2}, w, y, u^+,
w^+$ to denote $u^k, u^{k + 1 / 2}, w^k, y^k, u^{k + 1}, w^{k + 1}$. Define
\[ \tau \assign \min \{ \tfrac{s \sigma_2^2}{1296},
   \tfrac{\sigma_2^3 s^3}{36 \| p \|_1^2} \} . \]
Since $u^1 = w^1$, quadratic growth in \Cref{lem:reduced-pot} and the initialization condition imply
\[ f (u^1, w^1) - \zeta ( \us ) \leq 2 [\zeta (u^1) - \zeta ( \us )]
   \leq \min \{ \tfrac{s \sigma_2}{\| p \|_1}, 1 \} \tau . \]
We show by induction that
\[ \max \{ \zeta (u), \zeta (w) \} - \zeta ( \us ) \leq \tau,
   \quad f (u, w) - \zeta ( \us ) \leq \min \{ \tfrac{s \sigma_2}{\| p \|_1}, 1 \} \tau . \]
Both inequalities hold at initialization. Suppose they hold at the current iteration. The point $y
= u + \tfrac{\sqrt{\sigma_2}}{2 + \sqrt{\sigma_2}} (w - u)$ is a convex
combination of $u$ and $w$, so by convexity,
\[ \zeta (y) \leq \tfrac{\sqrt{\sigma_2}}{2 + \sqrt{\sigma_2}} \zeta (w) + ( 1 -
   \tfrac{\sqrt{\sigma_2}}{2 + \sqrt{\sigma_2}} ) \zeta (u) \leq \zeta ( \us ) + \tau . \]
By convexity and \Cref{lem:reduced-pot}, for all $t \in [0,
1]$, $\lambda_2 ( P^{- 1 / 2} \nabla^2 \zeta ( y + t (
\us - y ) ) P^{- 1 / 2} ) \geq \tfrac{\sigma_2}{2}$.
\begin{align}
  & \zeta ( \us ) - \zeta (y) - \langle \nabla \zeta (y), \us
  - y \rangle \nonumber\\
  ={} & \textstyle \int_0^1 (1 - t) \langle \us - y, \nabla^2 \zeta ( y + t
  ( \us - y ) ) ( \us - y ) \rangle \mathd t
  \nonumber\\
  ={} & \textstyle \int_0^1 (1 - t) \langle P^{1 / 2} \Pi ( \us - y ), P^{-
  1 / 2} \nabla^2 \zeta ( y + t ( \us - y ) ) P^{- 1 /
  2} P^{1 / 2} \Pi ( \us - y ) \rangle \mathd t \nonumber\\
  \geq{} & \textstyle\int_0^1 (1 - t) \tfrac{\sigma_2}{2} \| \Pi ( \us - y
  ) \|_P^2 \mathd t = \tfrac{\sigma_2}{4} \| \Pi ( \us
  - y ) \|_P^2 . \nonumber
\end{align}
Here we use $\Pi^{\top} \nabla^2 \zeta (u) \Pi = \nabla^2 \zeta (u)$ and $P^{1 / 2} \Pi (\us - y) \perp P^{1 / 2} \1_m$, since $\Pi$ removes a multiple of $\1_m$ and projects onto $p^\perp$. Next, if $\nabla \zeta (y) = 0$, then $u^{1 / 2} = y$. Otherwise, since $\zeta (y) - \zeta ( \us ) \leq \tau \leq \tfrac{s \sigma_2^2}{1296}$, \Cref{lem:hu-property} gives $\zeta (y - P^{- 1} \nabla \zeta (y)) \leq \zeta (y)$. By convexity, we have $\zeta
(u^{1 / 2}) \leq \zeta (y) \leq \zeta ( \us ) + \tau$. Hence $P^{-
1 / 2} \nabla^2 \zeta (y + t (u^{1 / 2} - y)) P^{- 1 / 2} \preceq 2 I$ for all $t \in
[0, 1]$, giving
\begin{align}
  & \zeta (u^{1 / 2}) - \zeta (y) - \langle \nabla \zeta (y), u^{1 / 2} - y
  \rangle \nonumber\\
  ={} & \textstyle \int_0^1 (1 - t) \langle u^{1 / 2} - y, \nabla^2 \zeta (y + t (u^{1 / 2}
  - y)) (u^{1 / 2} - y) \rangle \mathd t \leq \| \Pi (u^{1 / 2} - y) \|_P^2
  . \nonumber
\end{align}

Finally, by convexity, we have $\zeta (u) \geq \zeta (y) + \langle \nabla
\zeta (y), u - y \rangle$. Following the potential argument in \cite{d2021acceleration}, combining the three inequalities
\begin{align}
  \zeta ( \us ) - \zeta (y) - \langle \nabla \zeta (y), \us -
  y \rangle \geq{} & \tfrac{\sigma_2}{4} \| \Pi ( \us - y
  ) \|_P^2 \nonumber\\
  \zeta (u^{1 / 2}) - \zeta (y) - \langle \nabla \zeta (y), u^{1 / 2} - y
  \rangle \leq{} &  \| \Pi (u^{1 / 2} - y) \|_P^2 \nonumber\\
  \zeta (u) \geq{} & \zeta (y) + \langle \nabla \zeta (y), u - y \rangle
  \nonumber
\end{align}
with weights $\tfrac{1}{2} \sqrt{\sigma_2}, 1, 1 - \tfrac{1}{2} \sqrt{\sigma_2}$ and using $u - y = \tfrac{1}{2} \sqrt{\sigma_2} (y - w)$ gives
\begin{align}
  \zeta (u^{1 / 2}) - \zeta ( \us ) \leq{} & (1 - \tfrac{1}{2} \sqrt{\sigma_2}) [\zeta (u) - \zeta ( \us )] \nonumber\\
  & + \tfrac{1}{2} \sqrt{\sigma_2} \langle \nabla \zeta (y), (1 - \tfrac{1}{2} \sqrt{\sigma_2}) w + \tfrac{1}{2} \sqrt{\sigma_2} y - \us \rangle \nonumber\\
  & - \tfrac{\sigma_2^{3 / 2}}{8} \| \Pi (y - \us) \|_P^2 - \tfrac{1}{4} \| \nabla \zeta (y) \|_{P^{- 1}}^2 . \nonumber
\end{align}
Since $\langle \1_m, \nabla \zeta (y) \rangle = 0$, we have $\Pi P^{- 1} \nabla \zeta (y) = P^{- 1} \nabla \zeta (y)$. Expanding the squared norm in the potential gives
\begin{align}
  \tfrac{\sigma_2}{4} \| \Pi (w^+ - \us) \|_P^2 ={} & (1 - \tfrac{1}{2} \sqrt{\sigma_2}) \tfrac{\sigma_2}{4} \| \Pi (w - \us) \|_P^2
  + \tfrac{\sigma_2^{3 / 2}}{8} \| \Pi (y - \us) \|_P^2 \nonumber\\
  & - \tfrac{\sigma_2^{3 / 2}}{8} (1 - \tfrac{1}{2} \sqrt{\sigma_2}) \| \Pi (w - y) \|_P^2 + \tfrac{1}{4} \| \nabla \zeta (y) \|_{P^{- 1}}^2\nonumber\\
  & - \tfrac{1}{2} \sqrt{\sigma_2} \langle \nabla \zeta (y), (1 - \tfrac{1}{2} \sqrt{\sigma_2}) w + \tfrac{1}{2} \sqrt{\sigma_2} y - \us \rangle \nonumber\
\end{align}
Adding the two relations yields
\begin{align}
  f (u^{1 / 2}, w^+) - \zeta ( \us ) \leq{} & (1 - \tfrac{1}{2} \sqrt{\sigma_2}) [f (u, w) - \zeta ( \us )] - \tfrac{\sigma_2^{3 / 2}}{8} (1 - \tfrac{1}{2} \sqrt{\sigma_2}) \| \Pi (w - y) \|_P^2 \nonumber\\
  \leq{} & (1 - \tfrac{1}{2} \sqrt{\sigma_2}) [f (u, w) - \zeta ( \us )] . \nonumber
\end{align}
Using $\zeta (u^+) \leq \zeta (u^{1 / 2})$ gives the claimed contraction.
 Finally, we verify the induction hypotheses at the next iteration. The contraction implies
\[ \tfrac{\sigma_2}{4} \| \Pi (w^+ - \us) \|_P^2 \leq f (u^+, w^+) - \zeta ( \us )
   \leq f (u, w) - \zeta ( \us ) \leq \min \{ \tfrac{s \sigma_2}{\| p \|_1}, 1 \} \tau . \]
By invariance under shifts along $\1_m$, global $L$-smoothness, and $L = \tfrac{1}{2} \| p \|_1$, we have
\begin{align}
  \zeta (w^+) - \zeta ( \us ) \leq{} & \tfrac{L}{2} \| \Pi (w^+ - \us) \|^2 \nonumber\\
  \leq{} & \tfrac{L}{2s} \| \Pi (w^+ - \us) \|_P^2 \nonumber\\
  \leq{} & \tfrac{2L}{s \sigma_2} \min \{ \tfrac{s \sigma_2}{\| p \|_1}, 1 \} \tau \leq \tau . \nonumber
\end{align}
Also, $\zeta (u^+) \leq \zeta (u) \leq \zeta ( \us ) + \tau$, and the contraction preserves the bound on $f (u^+, w^+) - \zeta ( \us )$. This completes the induction and the proof.

\end{document}